\documentclass{article}

\usepackage{color}
\usepackage[utf8]{inputenc}
\usepackage[english]{babel}

\usepackage{enumitem}

\usepackage{amsmath}
\usepackage{amsfonts}
\usepackage{amssymb}
\usepackage{amsthm}
\usepackage{mathtools}
\usepackage{tikz-cd}
\usepackage{extarrows}

\makeatletter
\renewenvironment{proof}[1][\proofname]{\par
	\pushQED{\qed}%
	\normalfont \topsep\z@skip 
	\trivlist
	\item[\hskip\labelsep
	\itshape
	#1\@addpunct{.}]\ignorespaces
}{%
	\popQED\endtrivlist\@endpefalse
}
\makeatother

\usepackage[hypertexnames=False]{hyperref}
\usepackage{thmtools}
\usepackage[english,capitalise]{cleveref}

\theoremstyle{definition}
\newtheorem{Definition}{Definition}[section]
\newtheorem{Notation}[Definition]{Notation}

\theoremstyle{plain}
\newtheorem{Remark}[Definition]{Remark}

\newtheorem{Lemma}[Definition]{Lemma}

\newtheorem{Corollary}[Definition]{Corollary}

\newtheorem{Proposition}[Definition]{Proposition}
\newtheorem*{Proposition*}{Proposition}

\newtheorem{Theorem}[Definition]{Theorem}
\newtheorem*{Theorem*}{Theorem}

\newtheorem{introthm}{Theorem}

\newtheorem{introcor}[introthm]{Corollary}

\theoremstyle{definition}
\newtheorem{introdef}[introthm]{Definition}

\theoremstyle{definition}
\newtheorem{introex}[introthm]{Example}

\Crefname{Corollary}{Corollary}{Corollaries}
\Crefname{Lemma}{Lemma}{Lemmas}

\DeclareMathOperator{\id}{id}
\DeclareMathOperator{\op}{op}

\DeclareMathOperator{\QCoh}{QCoh}

\DeclareMathOperator{\Spec}{Spec}
\DeclareMathOperator{\Cart}{Cart}
\DeclareMathOperator{\Arr}{Arr}
\DeclareMathOperator{\Fun}{Fun}
\DeclareMathOperator{\ev}{ev}
\DeclareMathOperator{\Map}{Map}
\DeclareMathOperator{\Eq}{Eq}

\DeclareMathOperator{\Mod}{Mod}
\DeclareMathOperator{\Cat}{Cat}

\DeclareMathOperator{\inc}{inc}
\DeclareMathOperator{\LFun}{LFun}
\DeclareMathOperator{\tex}{t-ex}
\DeclareMathOperator{\ex}{ex}
\DeclareMathOperator{\lex}{lex}
\DeclareMathOperator{\rtex}{r-t-ex}

\DeclareMathOperator{\cofib}{cofib}
\DeclareMathOperator{\pr}{pr}
\DeclareMathOperator{\End}{End}
\DeclareMathOperator{\LEq}{LEq}
\DeclareMathOperator{\Frob}{Frob}

\newcommand\D{\mathcal{D}}

\newcommand\B{\mathcal{B}}
\newcommand\A{\mathcal{A}}

\newcommand\FF{\mathbb{F}}
\newcommand\C{\mathcal{C}}

\newcommand\ZZ{\mathbb{Z}}

\newcommand\DU{\mathcal{D}(U_{\mathcal{A}})}
\newcommand\UD{U_{\mathcal{D}(\mathcal{A})}}
\newcommand\DL{\mathcal{D}(L_{\mathcal{A}})}
\newcommand\LD{L_{\mathcal{D}(\mathcal{A})}}
\newcommand\heart{\heartsuit}
\newcommand\taug[1]{\tau_{\geq #1}}
\newcommand\iotag[1]{\iota_{\geq #1}}
\newcommand\taul[1]{\tau_{\leq #1}}
\newcommand\iotal[1]{\iota_{\leq #1}}
\newcommand{\laxtimes}[1]{\mathop{\times\mkern-13mu\raise1.3ex\hbox{$\scriptscriptstyle\to$}_{#1}}}
\newcommand*{\sheafhom}{\mathscr{H}\kern -.5pt om}

\renewcommand\O{\mathcal{O}}

\usepackage{todonotes}

\title{The derived \texorpdfstring{$\infty$}{Infinity}-category of Cartier Modules}
\author{Klaus Mattis\footnote{\href{www.klaus-mattis.com}{www.klaus-mattis.com}}\,\, and Timo Weiß\footnote{\href{mailto:timo.weiss@uni-mainz.de}{timo.weiss@uni-mainz.de}}}

\begin{document}

    \maketitle
    \begin{abstract}
        For an endofunctor $F\colon\C\to\C$ on an ($\infty$-)category $\C$ we 
        define the $\infty$-category $\Cart(\C,F)$ of generalized Cartier modules 
        as the lax equalizer of $F$ and the identity.
        This generalizes the notion of Cartier modules on $\FF_p$-schemes considered in \cite{BBCartmod}. 
        We show that in favorable cases $\Cart(\C,F)$ is monadic over $\C$. 
        If $\A$ is a Grothendieck abelian category and $F\colon\A\to\A$ is an exact and colimit-preserving endofunctor, 
        we use this fact to construct an equivalence $\D(\Cart(\A,F)) \simeq \Cart(\D(\A),\D(F))$ of stable $\infty$-categories. 
        We use this equivalence to construct a perverse t-structure on $\D(\Cart(\Mod(X), F_*))$ for any Noetherian $\FF_p$-scheme $X$
        with absolute Frobenius $F$.
        If $F$ is finite, this coincides with the perverse t-structure constructed in \cite{Baudin}.
    \end{abstract}
    \tableofcontents
    \newpage
    \section{Introduction}

A distinguishing aspect of algebraic geometry over a field of positive characteristic $p > 0$ is the presence of the Frobenius endomorphism. A particular point of interest are modules with an action of the Frobenius.

If $X$ is an $\FF_p$-scheme then a (quasi-coherent) $\O_X$-module with a left action of the absolute Frobenius $F\colon X \to X$ is called a Frobenius module. They are related via a Riemann--Hilbert-type correspondence to $p$-torsion étale sheaves, cf.\ \cite{EK, BL, BP}. 

There is also a dual notion of Cartier modules, which are related to Frobenius modules via Grothendieck--Serre duality by a result of Baudin \cite[Theorem 4.2.7]{Baudin}. They are (quasi-coherent) $\O_X$-modules with a right action of the Frobenius, or equivalently pairs $(M,\kappa_M)$ of an $\O_X$-module $M$ and an $\O_X$-linear morphism $\kappa_M\colon F_*M \to M$. Their category $\Cart(X)$ was first considered by Anderson in \cite{Anderson} and more thoroughly studied by Blickle and Böckle in \cite{BBCartmod, BBCartcrys}. 
Both Frobenius modules and Cartier modules are important tools in positive characteristic commutative algebra and algebraic geometry,
and have been studied by many authors \cite{lyubeznik1997f,lyubeznik2001commutation,schwede2009f,patakfalvi2016frobeniustechniquesbirationalgeometry,99bc2f14-cf95-3c44-a77d-e80ab3dcf24f}.

The aim of this paper is to introduce a general framework where the classical notion of Cartier modules (and Frobenius modules) can be viewed as a special instance:
For $\C$ any ($\infty$-)category and $F\colon \C \to \C$ any endofunctor we define the following ($\infty$-)category.
\begin{introdef}[\Cref{DefCart}]
    We define the $\infty$-category $\Cart(\C,F)$ of \textit{Cartier modules in $\C$ with respect to the functor $F$} as the pullback
    \[\begin{tikzcd}
        \Cart(\C,F) \ar{r}{U_{\C}} \ar{d}[swap]{\kappa_{\C}} & \C \ar{d}{(F,\id_{\C})} \\
        \Arr(\C) \ar{r}[swap]{(\ev_0,\ev_1)} & \C \times \C
    \end{tikzcd}\]
    in $\Cat_{\infty}$, where $\Arr(\C)\coloneqq\Fun(\Delta^1,\C)$ is the $\infty$-category of arrows in $\C$ and the functors $\ev_0, \ev_1\colon \Fun(\Delta^1,\C) \to \C$ are induced by the inclusions $\{0\} \hookrightarrow \Delta^1$ and $\{1\} \hookrightarrow \Delta^1$, respectively.
\end{introdef}

\begin{introex}
    As a special case, we recover the definition of Cartier modules in \cite{BBCartmod} as $\Cart(X)\coloneqq\Cart(\QCoh(X),F_*)$,
    where $F_*$ denotes the direct image functor along the absolute Frobenius $F$ on $X$.
\end{introex}

This general framework allows us to consider the $\infty$-category of Cartier modules in the derived $\infty$-category of quasicoherent modules, or more generally in the derived $\infty$-category $\D(\A)$ of any Grothendieck abelian category $\A$.
The main theorem of this paper is the following comparison result, which roughly states that the operations 
of taking Cartier modules and taking the derived $\infty$-category commute.
\begin{introthm}[\Cref{derived}]\label{mainthm}
    Let $\A$ be a Grothendieck abelian category and $F\colon \A \to \A$ an exact and colimit-preserving functor. In particular, $F$ induces a functor $\D(F)\colon \D(\A) \to \D(\A)$. Then $\Cart(\A,F)$ is a Grothendieck abelian category and there is a canonical functor $\D(\Cart(\A,F)) \to \Cart(\D(\A),\D(F))$ which is an equivalence of $\infty$-categories. Moreover, both $\infty$-categories carry a natural t-structure and the equivalence is t-exact.
\end{introthm}
There was strong evidence that suggested such a result, as there is a similar equivalence of stable $\infty$-categories 
\begin{equation*}
    \D(\Arr(\A)) \simeq \Arr(\D(\A)) \, ,
\end{equation*}
see e.g.\ the answer of Hoyois on mathoverflow \cite{OverflowHoyois}.
Note that this equivalence does not hold if one replaces the derived $\infty$-category $\D(-)$ 
by the ordinary derived category $h\D(-)$. In fact, there is not even a triangulated structure on $\Arr(h\D(\A))$, except for the trivial case $\A = 0$.
The reason for this is that $\Arr(h\D(X))$ is the category of `homotopy commutative' arrows in $\D(X)$, which is not equivalent to $h\Arr(\D(X))$, the category of `homotopy coherent' arrows in $\D(X)$, cf.\ \cite[Section 1.2.6]{HTT}. 

This `coherence problem' is why we have to consider (stable) $\infty$-categories instead of the triangulated theory:
One can still write down a comparison functor
\[\Phi \colon h\D(\Cart(X)) \to \Cart(h\D(X), F_*) \, ,\]
but this functor fails to be faithful.

Note that it is easy to map into the category $\Cart(h\D(X), F_*)$ (since it is defined as a limit),
but very hard to map into the category $h\D(\Cart(X))$ (since this is a localization).
This makes it hard to define functors into the ordinary derived category of Cartier modules 
(e.g.\ cohomological operations for Cartier modules as in \cite{BBCartcrys}), 
like for example the twisted inverse image functor $f^!$, which are not induced by corresponding functors on the level of abelian categories 
$\Cart(X) \to \Cart(Y)$. 

Our main theorem gives a convenient solution to this problem,
as it allows to easily construct functors into $\D(\Cart(X))$, which via the equivalence $\D(\Cart(X)) \simeq \Cart(\D(X))$ 
has a presentation as a limit of $\infty$-categories.
We will use this result in upcoming work \cite{us} where we construct a coherent functor formalism\footnote{We chose this name since the authors are currently unaware how many of the six functors exist but are confident that Cartier modules do \textit{not} admit a full six functor formalism.} of derived Cartier modules in appropriate geometric settings. This will in particular include a construction of the twisted inverse image functor $f^!$.

\subsection*{Consequences of the main theorem}
Specializing \Cref{mainthm}, we deduce the following:
If we apply it with $F=\id$, we recover the following well-known result about the derived $\infty$-category of the endomorphism category $\End(\A)$ of a Grothendieck abelian category $\A$ (cf.\ \cite[Definition 3.5]{BGT} for the definition of the endomorphism category of an $\infty$-category).
\begin{introcor}[\Cref{endmainthm}]
    There is a canonical equivalence of $\infty$-categories $\D(\End(\A)) \to \End(\D(\A))$.
\end{introcor}

On the other hand, if we apply \Cref{mainthm} with $\A=\QCoh(X)$ and $F=F_*$ the (exact) Frobenius pushforward functor, then we obtain the result about classical Cartier modules that was the motivation for this paper.

\begin{introcor}[\Cref{classicalCart}]
    There is a canonical equivalence of $\infty$-categories $\D(\Cart(X)) \to \Cart(\D(X),F_*)$.
\end{introcor}

Moreover, for $\A=\QCoh(X)$ and $F=F^*$ with $X$ a regular Noetherian $\FF_p$-scheme, we get a corresponding result about classical Frobenius modules.

\begin{introcor}[\Cref{cartadjoint,classicalFrob}]
    There is a canonical equivalence of $\infty$-categories $\D(\Frob(X)) \to \Cart(\D(X),F^*)$.
\end{introcor}

\subsection*{Monadicity of Cartier modules}
As a main ingredient for the proof of \Cref{mainthm} we prove the following result which could be of independent interest.

\begin{introthm}[\Cref{monadicity}]
    Let $\C$ be an $\infty$-category that admits countable coproducts and let $F\colon \C \to \C$ be an endofunctor that preserves them. Then the forgetful functor $U\colon \Cart(\C,F)\to \C$ exhibits $\Cart(\C,F)$ as monadic over $\C$ (in the sense of \cite[Definition 4.7.3.4]{HA}).
\end{introthm}

Recall that a monadic functor $U \colon \C \to \D$ abstracts the setting of the forgetful functor $U \colon \Mod_R \to \operatorname{Set}$
for a ring $R$. Precisely, a monadic functor $U$ is a conservative functor with a left adjoint (in the case of modules this is given by the free module functor),
such that $U$ commutes with certain sifted colimits.
Monadic functors are related to monads (i.e.\ monoid objects in an ($\infty$-)category of endofunctors with respect to the monoidal structure given by composition)
in the following way: Every adjunction $L \dashv R$ gives rise to a monad $RL$, with 
multiplication given by $RL \circ RL = R \circ LR \circ L \xrightarrow{\operatorname{counit}} RL$,
and unit $\operatorname{id} \xrightarrow{\operatorname{unit}} RL$. In particular, every monadic functor $U$ gives rise to a monad $UL$.
The Barr-Beck Theorem (cf.\ \cite[Theorem 4.7.3.5]{HA} for its formulation for $\infty$-categories) provides a partial converse to this construction: every monad $T \colon \C \to \C$ has a unique (up to contractible choice)
monadic functor $U \colon \D \to \C$ with left adjoint $L$ such that $T \simeq UL$.

In the proof of \Cref{mainthm}, we first show that both $\infty$-categories in question are monadic over the base $\infty$-category,
and then prove that the corresponding monads agree, thus giving an equivalence of the module-categories.

\subsection*{The perverse t-structure on Cartier modules}

Recall that there is a perverse t-structure on $\D(\Mod(X))$ for $X$ a Noetherian scheme, cf.\ \cite{BBD, Gabber, AB, bhatt2023applications}. 
Baudin constructed a corresponding perverse t-structure on
the $\infty$-category $\D^b(\Cart(\operatorname{Coh}(X), F_*))$ under the additional assumption that the absolute Frobenius on $X$ is finite \cite[Definition 5.2.1]{Baudin}. 
We use \Cref{mainthm} to give a more conceptual construction, that also removes this finiteness assumption.

\begin{introthm}[\Cref{perversethm,baudin}]
    Let $X$ be a Noetherian $\FF_p$-scheme.
    The perverse t-structure on $\D(\Mod(X))$ induces a perverse t-structure on $\D(\Cart(\Mod(X), F_*))$ such that the forgetful functor is t-exact for these t-structures.

    This t-structure coincides with the t-structure of \cite[Definition 5.2.1]{Baudin} in the case 
    where the absolute Frobenius of $X$ is finite, in the sense that the canonical functor 
    \begin{equation*}
        \D^b(\Cart(\operatorname{Coh}(X), F_*)) \to \D(\Cart(\Mod(X), F_*))
    \end{equation*}
    is t-exact.
\end{introthm}

\subsection*{Outline}

In \Cref{chap1} we define quite generally the $\infty$-category $\LEq(F,G)$ for functors $F,G\colon \C \to\D$; and its incarnations $\Cart(\C,F)$ and $\Frob(\C,F)$. Moreover, we show that properties of $\C$ and $\D$ like stability or presentability, together with appropriate conditions on $F$ and $G$, imply the same properties for $\LEq(F,G)$ (resp.\ for $\Cart(\C,F)$ or $\Frob(\C,F)$) and the forgetful functor ${U\colon \LEq(F,G) \to \C}$, cf.\ \Cref{leqproperties}. In particular, we prove the first part of \Cref{mainthm} that states that $\Cart(\A,F)$ is a Grothendieck abelian category.

In \Cref{chap2}, if $\C$ and $\D$ are stable $\infty$-categories with t-structures such that the functors $F,G\colon \C \to \D$ are exact and $F$ is right t-exact and $G$ is left t-exact, we introduce a t-structure on $\LEq(F,G)$ that is characterized by the property that $U \colon \LEq(F,G) \to \C$ is t-exact, cf.\ \Cref{tstructure}. Moreover, we identify the heart as $\LEq(F^{\heart},G^{\heart})$ in \Cref{heart}.

In \Cref{chap3} we explicitly construct the left adjoint of the forgetful functor $U\colon \Cart(\C,F) \to \C$ and derive t-exactness results about it, cf.\ \Cref{Lheart}. Moreover, we show that the existence of a left adjoint of $U$ implies that $U$ exhibits $\Cart(\C,F)$ as monadic over $\C$ if $\C$ is a presentable $\infty$-category and $F$ preserves colimits, cf.\ \Cref{monadicity}. This is one of the key ingredients for the proof of \Cref{mainthm}. 

In \Cref{chap4} we state and prove our main theorem, cf.\ \Cref{derived}. Roughly, by a consequence of the $\infty$-categorical Barr-Beck Theorem \cite[Theorem 4.7.3.5]{HA} it is enough to check the following:
\begin{enumerate}[label=(\arabic*)]
    \item\label{introbb1} The functor $\DU\colon \D(\Cart(\A,F)) \to \D(\A)$ induced by the forgetful functor $U_{\A}$ on $\Cart(\A,F)$ exhibits $\D(\Cart(\A,F))$ as monadic over $\D(\A)$.
    \item\label{introbb2} The forgetful functor $\UD\colon \Cart(\D(\A),\D(F)) \to \D(\A)$ exhibits the $\infty$-category $\Cart(\D(\A),\D(F))$ as monadic over $\D(\A)$.
    \item\label{introbb3} Denote by $\DL$ and $\LD$ the left adjoints of the functors above. Then the natural map $\UD\LD(M) \to \DU\DL(M)$ is an equivalence for each object $M$ of $\D(\A)$.
\end{enumerate}
The proof of \ref{introbb1} is an application of a more general result about induced functors on derived $\infty$-categories (\Cref{Dmonadic}) once we know that $U_{\A}$ exhibits $\Cart(\A,F)$ as monadic over $\A$. But this, as well as \ref{introbb2}, is already shown in \Cref{chap3}. The key point in the proof of \ref{introbb3} is a reduction to the case where $M$ is an object in the heart $\D(\A)^{\heart}$. To do so, we use the t-exactness properties of the forgetful functor and its left adjoint that we showed in the previous sections.

In \Cref{chap5} we specialize the discussion to the case where $\A$ is the category of $\O_X$-modules over a Noetherian $\FF_p$-scheme $X$. In particular, we construct a perverse t-structure on $\D(\Cart(\Mod(X),F_*))$.

Furthermore, in \Cref{appendix} we collect some essentially well-known results about the (unbounded) derived $\infty$-category $\D(\A)$ of a Grothendieck abelian category $\A$.

\subsection*{Notation}
We freely use the language of $\infty$-categories as developed in \cite{HTT, HA}. In particular, we regard an (ordinary) category as an $\infty$-category via the nerve construction. Moreover, we denote by $\Cat_{\infty}$ the (very large) $\infty$-category of (large) $\infty$-categories.

\subsection*{Acknowledgements}
We thank Manuel Blickle for suggesting the topic. We thank Tom Bachmann, Manuel Blickle, Anton Engelmann, Lorenzo Mantovani and Luca Passolunghi for helpful discussions and reading a draft of the article.

The authors acknowledge support by the Deutsche Forschungsgemeinschaft
(DFG, German Research Foundation) through the Collaborative Research Centre TRR 326 \textit{Geometry and Arithmetic of Uniformized Structures}, project number 444845124.
    \section{Definition of Cartier modules}\label{chap1}

Let $\C$ be an $\infty$-category and $F\colon \C \to \C$ an endofunctor. In this section we define the $\infty$-category $\Cart(\C,F)$ of Cartier modules in $\C$ with respect to the functor $F$ which is a generalization of the usual notion of Cartier modules. Moreover, we investigate how certain properties of $\C$, like stability or presentability, induce the same properties on $\Cart(\C,F)$. We do the same for the $\infty$-category $\Frob(\C,F)$ of Frobenius modules in $\C$ with respect to $F$ and for the more general concept of lax equalizers.

As $\Cart(\C,F)$ and $\Frob(\C,F)$ are special cases of lax equalizers recall the following definition.

\begin{Definition}[{\cite[Definition II.1.4]{TC}}]
    Let $\C$ and $\D$ be $\infty$-categories and $F,G\colon \C \to \D$ be functors. The \textit{lax equalizer of $F$ and $G$} is the $\infty$-category $\LEq(F,G)$ defined as the pullback
    \[\begin{tikzcd}
        \LEq(F,G) \ar{r}{U} \ar{d}[swap]{\kappa} & \C \ar{d}{(F,G)} \\
        \Arr(\D) \ar{r}[swap]{(\ev_0,\ev_1)} & \D \times \D
    \end{tikzcd}\]
    in $\Cat_{\infty}$, where $\Arr(\D)\coloneqq\Fun(\Delta^1,\D)$ is the $\infty$-category of arrows in $\D$ and the functors $\ev_0, \ev_1\colon \Fun(\Delta^1,\D) \to \D$ are induced by the inclusions $\{0\} \hookrightarrow \Delta^1$ and $\{1\} \hookrightarrow \Delta^1$, respectively.
\end{Definition}

\begin{Notation}
    Throughout the paper, in the situation above, we denote the natural maps by $U\colon \LEq(F,G) \to \C$ and $\kappa\colon \LEq(F,G) \to \Arr(\D)$, respectively. We refer to the functor $U$ as the \emph{forgetful functor}.
\end{Notation}

\begin{Remark}\label{leqsimplices}
    Using the universal property of the pullback we can describe the objects and morphisms of the $\infty$-category $\LEq(F,G)$:
\begin{enumerate}[label=(\alph*)]
    \item Objects of $\LEq(F,G)$ are given by pairs $(c,f)$ of an object $c$ of $\C$ and an arrow $f\colon F(c)\to G(c)$ in $\D$.
    \item A morphism $\varphi \colon (c, f) \to (d, g)$ in $\LEq(F,G)$ is given by the data of a morphism $U\varphi \colon c \to d$ in $\C$
    and a commutative diagram 
    \begin{center}
        \begin{tikzcd}
            F(c) \ar[r, "{F(U\varphi)}"] \ar[d, "f"'] &F(d) \ar[d, "g"] \\
            G(c) \ar[r, "{G(U\varphi)}"'] &G(d)\rlap{.}
        \end{tikzcd}
    \end{center} 
\end{enumerate}
    A formula for mapping spaces in $\LEq(F, G)$ can be found in \Cref{leqproperties}\ref{leqmap}.
\end{Remark}

\begin{Definition}\label{DefCart}
    Let $\C$ be an $\infty$-category and $F\colon \C \to \C$ be an endofunctor.
    We define the $\infty$-category $\Cart(\C,F)$ of \textit{Cartier modules in $\C$ with respect to the functor $F$} as $\LEq(F,\id_{\C})$ and the $\infty$-category $\Frob(\C,F)$ of \textit{Frobenius modules in $\C$ with respect to the functor $F$} as $\LEq(\id_{\C},F)$.
\end{Definition}

By \Cref{leqsimplices} we can describe the objects and morphisms of $\Cart(\C,F)$ and $\Frob(\C,F)$.

\begin{Remark}
    Let $\C$ be an $\infty$-category and $F\colon \C \to \C$ be an endofunctor.
\begin{enumerate}[label=(\alph*)]
    \item Objects of the $\infty$-category $\Cart(\C,F)$ are given by pairs $(M,\kappa_M)$ where $M \in \C$ is an object of $\C$ and $\kappa_M\colon F(M) \to M$ is a morphism in $\C$.
    
    In particular, this generalizes the usual notion of Cartier modules on an $\FF_p$-scheme $X$ (cf.\ \cite[Definition 2.1]{BBCartmod}) that are pairs $(M,\kappa_M)$ of an $\O_X$-module $M$ together with an $\O_X$-linear morphism $\kappa_M\colon F_{\ast}M \to M$ where $F\colon X\to X$ denotes the absolute Frobenius endomorphism of $X$.
    \item Similarly, objects of the $\infty$-category $\Frob(\C,F)$ are given by pairs $(M,\tau_M)$ where $M \in \C$ is an object of $\C$ and $\tau_M\colon M \to F(M)$ is a morphism in $\C$.
    
    This generalizes the usual notion of Frobenius modules on an $\FF_p$-scheme $X$ (cf.\ \cite[Remark 1.3.2]{BL}) that are pairs $(M,\tau_M)$ of an $\O_X$-module $M$ together with an $\O_X$-linear morphism $\tau_M\colon M \to F_{\ast}M$ where $F\colon X\to X$ denotes the absolute Frobenius endomorphism of $X$.
    \item A morphism $\varphi \colon (M, \kappa_M) \to (N, \kappa_N)$ in $\Cart(\C,F)$ is given by the data of a morphism $U\varphi \colon M \to N$ in $\C$
    and a commutative diagram 
    \begin{center}
        \begin{tikzcd}
            F(M) \ar[r, "{F(U\varphi)}"] \ar[d, "\kappa_M"'] &F(N) \ar[d, "\kappa_N"] \\
            M \ar[r, "U\varphi"'] &N\rlap{,}
        \end{tikzcd}
    \end{center}
    and similarly for $\Frob(\C, F)$.
\end{enumerate}
\end{Remark}

The following proposition collects a number of useful facts about this construction.

\begin{Proposition}\label{leqproperties}
    Let $\C$ and $\D$ be $\infty$-categories and $F,G\colon \C \to \D$ be functors.
    \begin{enumerate}[label=(\alph*)]
        \item\label{leqmap} Let $M, N \in \LEq(F,G)$. The mapping space $\Map_{\LEq(F,G)}(M,N)$ is given by the equalizer
        \[\begin{tikzcd}
            \Eq \Bigg(\Map_{\C}(UM,UN) \ar[shift left=0.75ex]{rr}{\kappa(M)^*G} \ar[shift right=0.75ex]{rr}[swap]{\kappa(N)_*F} &&[-10 pt] \Map_{\D}(FUM,GUN)\Bigg) \rlap{.}
        \end{tikzcd}\]
        \item\label{leqUconservative} The forgetful functor $U\colon \LEq(F,G) \to \C$ is conservative, i.e.\ a morphism $f$ in $\LEq(F,G)$ is an equivalence if and only if $U(f)$ is an equivalence in $\C$.
        \item\label{leqlimits} Let $p\colon K \to \LEq(F,G)$ be a diagram such that the composite diagram ${K \to \LEq(F,G) \to \C}$ admits a limit and this limit is preserved by the functor $G$. Then $p$ admits a limit and the functor $U\colon \LEq(F,G) \to \C$ preserves this limit.
        \item\label{leqcolim} Let $p\colon K \to \LEq(F,G)$ be a diagram such that the composite diagram ${K \to \LEq(F,G) \to \C}$ admits a colimit and this colimit is preserved by the functor $F$. Then $p$ admits a colimit and the functor $U\colon \LEq(F,G) \to \C$ preserves this colimit.
        \item\label{leqexact} If $\C$ and $\D$ are stable $\infty$-categories (cf.\ \cite[Definition 1.1.1.9]{HA}) and $F$ and $G$ are exact, then $\LEq(F,G)$ is a stable $\infty$-category and the functor $U\colon \LEq(F,G) \to \C$ is exact.
        \item\label{leqcolimits} If $\C$ is a presentable $\infty$-category (cf.\ \cite[Definition 5.5.0.1]{HTT}), $\D$ is an accessible $\infty$-category (cf.\ \cite[Definition 5.4.2.1]{HTT}), $F$ preserves colimits and $G$ is accessible (cf.\ \cite[Definition 5.4.2.5]{HTT}), then $\LEq(F,G)$ is a presentable $\infty$-category and the functor $U\colon \LEq(F,G) \to \C$ preserves colimits.
        \item\label{leqadditive} If $\C$ and $\D$ are additive $\infty$-categories (cf.\ \cite[Definition 2.6]{GGN}) and $F$ and $G$ are additive, then $\LEq(F,G)$ is an additive $\infty$-category and the functor $U\colon \LEq(F,G) \to \C$ is additive.
        \item\label{leqabelian} If $\C$ and $\D$ are abelian categories and $F$ and $G$ are exact, then $\LEq(F,G)$ is an abelian category and the functor $U\colon \LEq(F,G) \to \C$ is exact.
        \item\label{leqGrothab} If $\C$ and $\D$ are Grothendieck abelian categories (i.e.\ they are presentable abelian and filtered colimits are exact) and $F$ and $G$ are exact and preserve colimits, then $\LEq(F,G)$ is a Grothendieck abelian category and the functor $U\colon \LEq(F,G) \to \C$ is exact and preserves colimits.
    \end{enumerate}
\end{Proposition}

\begin{proof}
    Parts \ref{leqmap} to \ref{leqcolimits} are \cite[Proposition II.1.5]{TC}. Note that \ref{leqcolim} follows from their proof of \ref{leqcolimits}.
    
    For the proof of the remaining parts first note that all the claims about the functor $U$ immediately follow from \ref{leqlimits} and \ref{leqcolim}.
    
    Using this there are only two things left to show in \ref{leqadditive}: that finite coproducts and finite products in $\LEq(F,G)$ are equivalent and that for each object $M$ of $\LEq(F,G)$ the shear map 
    \[M\oplus M \xrightarrow{\left(\begin{smallmatrix} 1&1\\0&1 \end{smallmatrix}\right)} M\oplus M\] 
    is an equivalence. But by \ref{leqUconservative} both of these can be checked after applying $U$ where it is true because $\C$ is additive and $U$ preserves finite coproducts and finite products (and in particular the shear map). This proves \ref{leqadditive}. 
    
    A similar argument, using finite colimits and finite limits instead of finite (co)products, and the comparison map between the coimage and the image, shows \ref{leqabelian} (it is clear that a limit of $1$-categories is still a $1$-category, see e.g.\ \cite[Proposition 2.3.4.12(4)]{HTT}). 
    
    For \ref{leqGrothab} we have to show that $\LEq(F,G)$ is presentable and that filtered colimits in $\LEq(F,G)$ are exact. Presentability was already shown in \ref{leqcolimits}. Checking that filtered colimits are exact can again be done after applying the functor $U$ because it is conservative by \ref{leqUconservative}, exact by \ref{leqabelian} and colimit-preserving by \ref{leqcolimits}. As filtered colimits in $\C$ are exact by assumption this completes the proof.
\end{proof}

Applying this with $\D=\C$ and $F$ or $G$ the identity we immediately get analogous results for Cartier modules and Frobenius modules.

\begin{Corollary}\label{properties}
    Let $\C$ be an $\infty$-category and $F\colon \C \to \C$ be an endofunctor.
    \begin{enumerate}[label=(\alph*)]
        \item\label{map} Let $M, N \in \Cart(\C,F)$. The mapping space $\Map_{\Cart(\C,F)}(M,N)$ is given by the equalizer
        \[\begin{tikzcd}
            \Eq \Bigg(\Map_{\C}(UM,UN) \ar[shift left=0.75ex]{rr}{\kappa(M)^*} \ar[shift right=0.75ex]{rr}[swap]{\kappa(N)_*F}  &&[-10 pt] \Map_{\C}(FUM,UN)\Bigg) \rlap{.}
        \end{tikzcd}\]
        \item\label{Uconservative} The forgetful functor $U\colon \Cart(\C,F) \to \C$ is conservative.
        \item\label{limits} Let $p\colon K \to \Cart(\C,F)$ be a diagram such that the composite diagram $K \to \Cart(\C,F) \to \C$ admits a limit. Then $p$ admits a limit and the functor $U\colon \Cart(\C,F) \to \C$ preserves this limit.
        \item\label{colim} Let $p\colon K \to \Cart(\C,F)$ be a diagram such that the composite diagram $K \to \Cart(\C,F) \to \C$ admits a colimit and this colimit is preserved by the functor $F$. Then $p$ admits a colimit and the functor $U\colon \Cart(\C,F) \to \C$ preserves this colimit.
        \item\label{exact} If $\C$ is a stable $\infty$-category and $F$ is exact, then $\Cart(\C,F)$ is a stable $\infty$-category and the functor $U\colon \Cart(\C,F) \to \C$ is exact.
        \item\label{colimits} If $\C$ is a presentable $\infty$-category and $F$ preserves colimits, then $\Cart(\C,F)$ is a presentable $\infty$-category and the functor $U\colon \Cart(\C,F) \to \C$ preserves colimits.
        \item\label{additive} If $\C$ is an additive $\infty$-category and $F$ is additive, then $\Cart(\C,F)$ is an additive $\infty$-category and the functor $U\colon \Cart(\C,F) \to \C$ is additive.
        \item\label{abelian} If $\C$ is an abelian category and $F$ is exact, then $\Cart(\C,F)$ is an abelian category and the functor $U\colon \Cart(\C,F) \to \C$ is exact.
        \item\label{Grothab} If $\C$ is a Grothendieck abelian category and $F$ is exact and preserves colimits, then $\Cart(\C,F)$ is a Grothendieck abelian category and the functor $U\colon \Cart(\C,F) \to \C$ is exact and preserves colimits.
    \end{enumerate}
\end{Corollary}

\begin{Corollary}\label{frobproperties}
    Let $\C$ be an $\infty$-category and $F\colon \C \to \C$ be an endofunctor.
    \begin{enumerate}[label=(\alph*)]
        \item\label{frobmap} Let $M, N \in \Frob(\C,F)$. The mapping space $\Map_{\Frob(\C,F)}(M,N)$ is given by the equalizer
        \[\begin{tikzcd}
            \Eq \Bigg(\Map_{\C}(UM,UN) \ar[shift left=0.75ex]{rr}{\kappa(M)^*F} \ar[shift right=0.75ex]{rr}[swap]{\kappa(N)_*}  &&[-10 pt] \Map_{\C}(FUM,UN)\Bigg) \rlap{.}
        \end{tikzcd}\]
        \item\label{frobUconservative} The forgetful functor $U\colon \Frob(\C,F) \to \C$ is conservative.
        \item\label{froblimits} Let $p\colon K \to \Frob(\C,F)$ be a diagram such that the composite diagram $K \to \Frob(\C,F) \to \C$ admits a limit and $F$ preserves this limit. Then $p$ admits a limit and the functor $U\colon \Frob(\C,F) \to \C$ preserves this limit.
        \item\label{frobcolim} Let $p\colon K \to \Frob(\C,F)$ be a diagram such that the composite diagram $K \to \Frob(\C,F) \to \C$ admits a colimit. Then $p$ admits a colimit and the functor $U\colon \Frob(\C,F) \to \C$ preserves this colimit.
        \item\label{frobexact} If $\C$ is a stable $\infty$-category and $F$ is exact, then $\Frob(\C,F)$ is a stable $\infty$-category and the functor $U\colon \Frob(\C,F) \to \C$ is exact.
        \item\label{frobcolimits} If $\C$ is a presentable $\infty$-category and $F$ is accessible, then $\Frob(\C,F)$ is a presentable $\infty$-category and the functor $U\colon \Frob(\C,F) \to \C$ preserves colimits.
        \item\label{frobadditive} If $\C$ is an additive $\infty$-category and $F$ is additive, then $\Frob(\C,F)$ is an additive $\infty$-category and the functor $U\colon \Frob(\C,F) \to \C$ is additive.
        \item\label{frobabelian} If $\C$ is an abelian category and $F$ is exact, then $\Frob(\C,F)$ is an abelian category and the functor $U\colon \Frob(\C,F) \to \C$ is exact.
        \item\label{frobGrothab} If $\C$ is a Grothendieck abelian category and $F$ is exact and preserves colimits, then $\Frob(\C,F)$ is a Grothendieck abelian category and the functor $U\colon \Frob(\C,F) \to \C$ is exact and preserves colimits.
    \end{enumerate}
\end{Corollary}

We conclude this section by investigating the relation between Cartier modules and Frobenius modules with respect to adjoint functors.

\begin{Proposition}\label{leqadjoint}
    Let $\C$ and $\D$ be $\infty$-categories and $F,G\colon \C \to \D$ be functors. If $F$ has a right adjoint $R\colon\D\to \C$ then there is an equivalence of $\infty$-categories 
    \[\LEq(F,G) \simeq \LEq(\id_{\C},RG) = \Frob(\C,RG) \, .\]
    If $G$ has a left adjoint $L\colon\D\to\C$ then there is an equivalence of $\infty$-categories 
    \[\LEq(F,G) \simeq \LEq(LF,\id_{\C}) = \Cart(\C,LF) \, .\]
\end{Proposition}

\begin{proof}
    Recall from \cite[Section 2.1]{LT} that the oriented fiber product $\C \laxtimes{F,\D,G} \C$ of $F$ and $G$ is defined as the pullback
    \[\begin{tikzcd}
        \C \laxtimes{F,\D,G} \C \ar{r} \ar{d} & \C \times \C \ar{d}{(F,G)} \\
        \Arr(\D) \ar{r}[swap]{(\ev_0, \ev_1)} & \D \times \D
    \end{tikzcd}\]
    in $\Cat_{\infty}$. Hence, there is a commutative diagram
    \begin{equation}\label{pullbackpaste}
        \begin{tikzcd}
        \LEq(F,G) \ar{r}{U} \ar{d} & \C \ar{d}{\Delta} \\
        \C \laxtimes{F,\D,G} \C \ar{r} \ar{d} & \C \times \C \ar{d}{(F,G)} \\
        \Arr(\D) \ar{r}[swap]{(\ev_0, \ev_1)} & \D \times \D
    \end{tikzcd}
    \end{equation}
    in $\Cat_{\infty}$, where $\Delta$ denotes the diagonal map. By definition of the lax equalizer $\LEq(F,G)$ and the oriented fiber product $\C \laxtimes{F,\D,G} \C$ the outer rectangle and the lower square of (\ref{pullbackpaste}) are cartesian diagrams. Thus, by the pasting law for pullbacks (cf.\ the dual of \cite[Lemma 4.4.2.1]{HTT}) the upper square of (\ref{pullbackpaste}) is also a cartesian diagram.

    If $F$ has a right adjoint functor $R$, by \cite[Lemma 2.2]{LT} there is a commutative diagram
    \[\begin{tikzcd}
        \C \laxtimes{F,\D,G} \C \ar{rr}{\simeq} \ar{dr} && \C \laxtimes{\id_{C},\C,RG} \C \ar{dl} \\
        & \C \times \C &
    \end{tikzcd}\]
    where the horizontal arrow is an equivalence of $\infty$-categories. Combining this with the discussion above we get that
    \begin{align*}
        \LEq(F,G) &\simeq \lim\left(\C \xrightarrow{\Delta} \C\times\C \leftarrow \C \laxtimes{F,\D,G} \C \right) \\
        &\simeq \lim\left(\C \xrightarrow{\Delta} \C\times\C \leftarrow \C \laxtimes{\id_{C},\C,RG} \C \right) \\
        &\simeq \LEq(\id_C, RG) \, .
    \end{align*}
    In the case where $G$ has a left adjoint one argues analogously.
\end{proof}

\begin{Corollary}\label{cartadjoint}
    Let $\C$ be an $\infty$-category and $L,R\colon \C \to \C$ be functors such that $L$ is left adjoint to $R$. 
    Then there is an equivalence of $\infty$-categories 
    \begin{equation*}
        \Cart(\C,L) \simeq \Frob(\C,R).
    \end{equation*}
\end{Corollary}
    \section{The induced t-structure on lax equalizers}\label{chap2}

We fix the following notation for the whole section: 

Let $\C$ be a stable $\infty$-category, cf.\ \cite[Definition 1.1.1.9]{HA}. We denote the shift functor by $\Sigma\colon \C \to \C$. On objects $X\in\C$ it is given by $\Sigma (X) = 0 \amalg_{X} 0$, cf.\ \cite[Notation 1.1.2.7]{HA}.

Moreover, assume that $\C$ carries a t-structure $(\C_{\geq 0}, \C_{\leq 0})$, cf.\ \cite[Definition 1.2.1.1]{HA}. 
This means that the following axioms are satisfied:
\begin{enumerate}[label=(\alph*)]
    \item $\C_{\geq 0}$ and $\C_{\leq 0}$ are subcategories of $\C$. 
    \item For $X \in \C_{\geq 0}$ and $Y \in \C_{\leq 0}$ we have $\Map_{\C} (X, \Sigma^{-1} Y) \simeq 0$.
    \item $\Sigma \C_{\geq 0} \subseteq \C_{\geq 0}$ and $\Sigma^{-1} \C_{\leq 0} \subseteq \C_{\leq 0}$.
    \item For each $X \in \C$ there exists a fiber sequence $X' \to X \to X''$ with $X' \in \C_{\geq 0}$ and $X'' \in \C_{\leq -1}$.
\end{enumerate}
Note that we use homological notation for t-structures. As usual, for each $n$ we denote by $\iotag{n}\colon \C_{\geq n} \rightleftarrows \C \colon\taug{n}$ the adjunction of the inclusion and the connective cover functor and by $\taul{n}\colon \C \rightleftarrows \C_{\leq n} \colon\iotal{n}$ the adjunction of the truncation and the inclusion functor, cf.\ \cite[Proposition 1.2.1.5, Notation 1.2.1.7]{HA}. 
In the following, for simplicity we write $\taug{n}$ for the composition $\iotag{n}\taug{n}$ and similarly for the truncations.
In particular, by \cite[Remark 1.2.1.8]{HA} for every $X \in \C$ there is a fiber sequence
\[ \taug{n} X \to X \to \taul{n-1} X \, .\]
Recall that the heart $\C^{\heart}$ of the t-structure is given by $\C_{\geq 0} \cap \C_{\leq 0}$ and that there are homotopy object functors $\pi_n \coloneqq \taug{0} \taul{0} \Sigma^{-n} \colon \C \to \C^{\heart}$, cf.\ \cite[Definition 1.2.1.11]{HA}.

Note that the $\infty$-category $\Arr(\C)$ is also stable by \cite[Proposition 1.1.3.1]{HA} and carries an induced t-structure given by $\Arr(\C)_{\geq 0} = \Arr(\C_{\geq 0})$ and $\Arr(\C)_{\leq 0} = \Arr(\C_{\leq 0})$. In particular, its heart is given by $\Arr(\C)^{\heart} = \Arr(\C^{\heart})$.

Furthermore, let $\D$ be a stable $\infty$-category and $F,G\colon \C \to \D$ be exact functors. If $\D$ carries a t-structure $(\D_{\geq 0},\D_{\leq 0})$ recall that $F$ is called right t-exact (resp.\ left t-exact) if $F(\C_{\geq 0}) \subseteq \D_{\geq 0}$ (resp.\ $F(\C_{\leq 0}) \subseteq \D_{\leq 0}$).

\ \\

In this section we introduce an induced t-structure on $\LEq(F,G)$. Note that this $\infty$-category is stable by \Cref{leqproperties}\ref{leqexact}. It is defined in such a way that the forgetful functor is automatically t-exact.

\begin{Definition}
    Define full subcategories $\LEq(F,G)_{\geq 0}$ and $\LEq(F,G)_{\leq 0}$ of $\LEq(F,G)$ as follows: An object $M \in \LEq(F,G)$ is in $\LEq(F,G)_{\geq 0}$ (resp.\ $\LEq(F,G)_{\leq 0}$) if and only if $UM$ is in $\C_{\geq 0}$ (resp.\ in $\C_{\leq 0}$).
\end{Definition}

We prove in \Cref{tstructure} that this defines a t-structure on $\LEq(F,G)$ if $F$ is right t-exact and $G$ is left t-exact. For the proof we need the following.

\begin{Lemma}\label{righttexactconncover}
    Let $(\D_{\geq 0},\D_{\leq 0})$ be a t-structure on $\D$ such that $F\colon \C \to \D$ is right t-exact. Then there is a commutative diagram
     \[\begin{tikzcd}
         F\taug{0} \ar{r}{Fc} \ar{d} & F \\
         \taug{0}F \ar{ur}[swap]{cF}
     \end{tikzcd}\]
     where $c\colon \taug{0} \to \id$ denotes the counit of the adjunction.
     
     Dually, if $G\colon\C \to \D$ is left t-exact then there is a commutative diagram
     \[\begin{tikzcd}
         \taug{0}G \ar{r}{cG} \ar{d} & G \\
         G\taug{0} \ar{ur}[swap]{Gc} \rlap{.}
     \end{tikzcd}\]
\end{Lemma}

\begin{proof}
    We only show the first claim. The second claim can be shown by a dual argument. There is a fiber sequence 
    \[\taug{0}F\taug{0} \xrightarrow{c} F\taug{0} \xrightarrow{u} \taul{-1}F\taug{0}\]
    where the last term is $0$ by the right t-exactness of $F$. Hence, the counit induces an equivalence $\taug{0}F\taug{0} \simeq F\taug{0}$.
    
    Moreover, the counit also yields a map $\taug{0}F\taug{0} \xrightarrow{\taug{0}Fc} \taug{0}F$. Putting everything together we obtain a diagram
    \[\begin{tikzcd}
        F\taug{0} \ar{r}{Fc} & F\\
        \taug{0}F\taug{0} \ar{u}{cF\taug{0}}[swap]{\simeq} \ar{d}[swap]{\taug{0}Fc} & \\
        \taug{0}F \ar[bend right]{uur}[swap]{cF}
    \end{tikzcd}\]
    that commutes because of the naturality of $c$.
\end{proof}

\begin{Proposition}\label{tstructure}
    Assume that there is a t-structure $(\D_{\geq 0}, \D_{\leq 0})$ on $\D$ such that $F\colon \C\to\D$ is right t-exact and $G\colon\C\to\D$ is left t-exact. Then the full subcategories $\LEq(F,G)_{\geq 0}$ and $\LEq(F,G)_{\leq 0}$ define a t-structure on $\LEq(F,G)$.

    Moreover, the forgetful functor $U\colon\LEq(F,G)\to\C$ is t-exact.
\end{Proposition}

\begin{proof}
    We check that $(\LEq(F,G)_{\geq 0}, \LEq(F,G)_{\leq 0})$ satisfies the properties of a t-structure. 
    First, we have to see that $\Sigma\LEq(F,G)_{\geq 0} \subseteq \LEq(F,G)_{\geq 0}$ and $\Sigma^{-1}\LEq(F,G)_{\leq 0} \subseteq \LEq(F,G)_{\leq 0}$. Let $M \in \LEq(F,G)_{\geq 0}$, i.e.\ $UM \in \C_{\geq 0}$. By \Cref{leqproperties}\ref{leqexact} the functor $U$ is exact, so we have $U\Sigma M \simeq \Sigma U M \in \C_{\geq 0}$ and hence $\Sigma M \in \LEq(F,G)_{\geq 0}$. The other inclusion can be shown analogously.
    
    Next, let $M\in \LEq(F,G)_{\geq 0}$ and $N \in \LEq(F,G)_{\leq -1}$. We have to show that $\Map_{\LEq(F,G)}(M,N)=0$. By \Cref{leqproperties}\ref{leqmap} this mapping space can be computed as the equalizer of two maps $\Map_{\C} (UM, UN) \rightrightarrows \Map_{\D}(FUM, GUN)$. But we have $UM \in \C_{\geq 0}$ and $UN \in \C_{\leq -1}$ by definition and $FUM \in \D_{\geq 0}$ by the right t-exactness of $F$ and $GUN \in \D_{\leq -1}$ by the left t-exactness of $G$. Thus, it follows that $\Map_{\C} (UM, UN)=0$ and $\Map_{\D}(FUM, GUN)=0$, hence $\Map_{\LEq(F,G)}(M,N)=0$.
    
    It remains to show that for any object $M \in \LEq(F,G)$ there is a fiber sequence $M' \to M \to M''$ with $M' \in \LEq(F,G)_{\geq 0}$ and $M'' \in \LEq(F,G)_{\leq -1}$. We have a fiber sequence 
    \begin{equation}\label{fibseq}
    \tau_{\geq 0} U M \xrightarrow{c} UM \xrightarrow{u} \tau_{\leq -1}UM
    \end{equation}
    in $\C$ where the maps are given by the counit resp.\ unit of adjunction. Using that $F$ is right t-exact and $G$ is left t-exact we can define a morphism 
    \[\kappa_{\geq 0}(M)\colon F\tau_{\geq 0} UM \to \tau_{\geq 0} F UM \xrightarrow{\kappa(M)} \tau_{\geq 0} GU M \to G\taug{0}UM\]
    where the first and last morphisms are the ones that we constructed in the proof of \Cref{righttexactconncover}. Equipping the left-hand term of (\ref{fibseq}) with this morphism we can promote it to an object of $\LEq(F,G)$. Furthermore, the left-hand map in (\ref{fibseq}) is a morphism in $\LEq(F,G)$ as can be seen from the commutative diagram
    \[\begin{tikzcd}
        F\tau_{\geq 0} U M \ar{r}{Fc} \ar{d} & FU M \ar{ddd}{\kappa(M)}\\
        \tau_{\geq 0} FUM \ar{ur}[swap]{cF} \ar{d}[swap]{\kappa(M)} & \\
        \tau_{\geq 0} GU M \ar{dr}{cG} \ar{d}& \\
        G\taug{0}UM \ar{r}[swap]{Gc} & U_{\C}M \rlap{.}
    \end{tikzcd}\]
    Here, the middle part of the diagram commutes because of the naturality of $c$ and the upper and lower parts commute by \Cref{righttexactconncover}.
    
    So we set $M'\coloneqq (\tau_{\geq 0} U M, \kappa_{\geq 0}(M))$ and $M'' \coloneqq \cofib(M' \xrightarrow{c} M)$ where the cofiber is taken in the category $\LEq(F,G)$. Since $\LEq(F,G)$ is stable this gives us a fiber sequence $M' \to M \to M''$ in $\LEq(F,G)$. As $U$ is exact by \Cref{leqproperties}\ref{leqexact}, we have 
    \[U M'' \simeq \cofib(\tau_{\geq 0} U M \xrightarrow{c} UM) \simeq \tau_{\leq -1}UM\]
    which shows that $M'' \in \LEq(F,G)_{\leq -1}$. Moreover, note that $M' \in \LEq(F,G)_{\geq 0}$ since $UM' = \taug{0}UM$.

    The t-exactness of the forgetful functor is clear.
\end{proof}

We can also compute the heart of the induced t-structure.

\begin{Proposition}\label{heart}
      Assume that there is a t-structure $(\D_{\geq 0}, \D_{\leq 0})$ on $\D$ such that $F\colon \C\to\D$ is right t-exact and $G\colon\C\to\D$ is left t-exact. The heart of the t-structure $(\LEq(F,G)_{\geq 0}, \LEq(F,G)_{\leq 0})$ is given by $\LEq(F,G)^{\heart} \simeq \LEq(F^{\heart},G^{\heart})$ where $F^{\heart}$ resp.\ $G^{\heart}$ denotes the composition $\C^{\heart} \hookrightarrow \C \xrightarrow{F} \D \xrightarrow{\pi_0} \D^{\heart}$ and for $G$ respectively.
\end{Proposition}

\begin{proof} 
    The forgetful functor $U\colon \LEq(F,G) \to \C$ is t-exact by \Cref{tstructure}, so it restricts to a functor $U^{\heart}\colon \LEq(F,G)^{\heart} \to \C^{\heart}$ of the hearts. Moreover, we denote by $U_{\C^{\heart}}\colon \LEq(F^{\heart},G^{\heart}) \to \C^{\heart}$ the forgetful functor of $\LEq(F^{\heart},G^{\heart})$.

    We define a functor $\Phi\colon \LEq(F,G)^{\heart} \to \LEq(F^{\heart},G^{\heart})$ via the universal property of the pullback and the commutative diagram
    \[\begin{tikzcd}
        \LEq(F,G)^{\heart} \ar[bend left]{drr}{U^{\heart}} \ar[dashed]{dr}{\Phi} \ar{dd}[swap]{\kappa|_{\LEq(F,G)^{\heart}}} & & \\
        & \LEq(F^{\heart},G^{\heart}) \ar{r}{U_{\C^{\heart}}} \ar{d}[swap]{\kappa_{\C^{\heart}}} & \C^{\heart} \ar{d}{(F^{\heart},G^{\heart})} \\
        \Arr(\D) \ar{r}[swap]{\pi_0}& \Arr(\D^{\heart}) \ar{r}[swap]{(\ev_0,\ev_1)} & \D^{\heart} \times \D^{\heart} \rlap{.}
    \end{tikzcd}\]
    We show that $\Phi$ is an equivalence by proving that it is essentially surjective and fully faithful.

    For the essential surjectivity let $M \in \LEq(F^{\heart},G^{\heart})$. Define a morphism $\tilde{\kappa}_M\colon FU_{\C^{\heart}}M \to GU_{\C^{\heart}}M$ to be the composition
    \[\begin{tikzcd}
        FU_{\C^{\heart}}M \ar{r}{\tilde{\kappa}_M} \ar{d}[swap]{u} & GU_{\C^{\heart}}M \\
        \taul{0} FU_{\C^{\heart}}M \ar{d}[swap]{\simeq} &  \taug{0} G U_{\C^{\heart}} M \ar{u}[swap]{c} \\
        F^{\heart} U_{\C^{\heart}} M \ar{r}{\kappa_{\C^{\heart}}(M)} & G^{\heart} U_{\C^{\heart}} M \ar{u}[swap]{\simeq}
    \end{tikzcd}\]
    where the arrows $u$ and $c$ are the unit resp.\ counit of adjunction, and the equivalences exist because $F$ is right t-exact and $G$ is left t-exact. Equipping $U_{\C^{\heart}}M$ with this morphism we can promote it to an object $(U_{\C^{\heart}}M, \tilde{\kappa}_M)$ of $\LEq(F,G)^{\heart}$. As $\pi_0(\tilde{\kappa}_M) \simeq \pi_0(\kappa_{\C^{\heart}}(M)) \simeq \kappa_{\C^{\heart}}(M)$, the functor $\Phi$ maps $(U_{\C^{\heart}}M, \tilde{\kappa}_M)$ to $(U_{\C^{\heart}}M,\kappa_{\C^{\heart}}(M)) = M$. Thus, the functor $\Phi$ is essentially surjective.

    To prove fully faithfulness of $\Phi$, let $M,N \in \LEq(F,G)^{\heart}$. We show that the map
    \[\Map_{\LEq(F,G)}(M,N) \to \Map_{\LEq(F^{\heart},G^{\heart})}(\Phi M, \Phi N)\]
    induced by $\Phi$ is an equivalence by using the description of the mapping spaces from \Cref{leqproperties}\ref{leqmap}. Recall that $\Map_{\LEq(F,G)}(M,N)$ is given by
    \[\begin{tikzcd}
            \Eq \Bigg(\Map_{\C}(UM,UN) \ar[shift left=0.75ex]{rr}{\kappa(M)^*G} \ar[shift right=0.75ex]{rr}[swap]{\kappa(N)_*F}  && \Map_{\D}(FUM,GUN)\Bigg)
    \end{tikzcd}\]
    and that $\Map_{\LEq(F^{\heart},G^{\heart})}(\Phi M,\Phi N)$ is given by
    \[\begin{tikzcd}
            \Eq \Bigg(\Map_{\C^{\heart}}(U^{\heart}M,U^{\heart}N) \ar[shift left=0.75ex]{rr}{\pi_0(\kappa(M))^*G^{\heart}} \ar[shift right=0.75ex]{rr}[swap]{\pi_0(\kappa(N))_*F^{\heart}}  &&[10 pt] \Map_{\D^{\heart}}(F^{\heart}U^{\heart}M,G^{\heart}U^{\heart}N)\Bigg) \rlap{.}
    \end{tikzcd}\]
    There is a commutative diagram
    \[\begin{tikzcd}
        \Map_{\C}(UM,UN) \ar{r}{\simeq} \ar{d}[swap]{G} & \Map_{\C^{\heart}}(U^{\heart}M,U^{\heart}N) \ar{d}{G^{\heart}} \\
        \Map_{\D}(GUM,GUN) \ar{r}{\pi_0} \ar{d}[swap]{\kappa(M)^*} & \Map_{\D^{\heart}}(G^{\heart}U^{\heart}M,G^{\heart}U^{\heart}N) \ar{d}{\pi_0(\kappa(M))^*} \\
        \Map_{\D}(FUM,GUN) \ar{r}{\pi_0} & \Map_{\D^{\heart}}(F^{\heart}U^{\heart}M,G^{\heart}U^{\heart}N) \rlap{.}
    \end{tikzcd}\]
    A similar diagram can be drawn for the map $\kappa(N)_*F$. Putting both diagrams together we get a commutative diagram
    \[\begin{tikzcd}
        \Map_{\C}(UM,UN) \ar[shift left=0.75ex]{rrr}{\kappa(M)^*G} \ar[shift right=0.75ex]{rrr}[swap]{\kappa(N)_*F} \ar{d}{\simeq} &&& \Map_{\D}(FUM,GUN) \ar{d}{\pi_0}\\
        \Map_{\C^{\heart}}(U^{\heart}M,U^{\heart}N) \ar[shift left=0.75ex]{rrr}{\pi_0(\kappa(M))^*G^{\heart}} \ar[shift right=0.75ex]{rrr}[swap]{\pi_0(\kappa(N))_*F^{\heart}}  &&& \Map_{\D^{\heart}}(F^{\heart}U^{\heart}M,G^{\heart}U^{\heart}N) \rlap{.}
    \end{tikzcd}\]
    Taking the equalizer of the rows, the vertical maps in the diagram induce the map $\Phi$. Thus, it suffices to show that the vertical arrow on the right-hand side is an equivalence. This is the case as it is given by the chain of equivalences 
    \begin{align*}
        \Map_{\D}(FUM,GUN) &\simeq \Map_{\D}(\taul{0}FUM,GUN) \\
        &\simeq \Map_{\D}(\taul{0}FUM,\taug{0} GUN) \\
        &\simeq \Map_{\D^{\heart}}(F^{\heart}U^{\heart}M,G^{\heart}U^{\heart}N)
    \end{align*}
    where we used the left t-exactness of $G$ and the adjunction $\taul{0} \dashv \iotal{0}$ in the first equivalence, the right t-exactness of $F$ and the adjunction $\iotag{0} \dashv \taug{0}$ in the second equivalence and both t-exactness properties in the last equivalence.
\end{proof}

From the definition of the t-structure and the description of its heart we immediately get:

\begin{Corollary}\label{Utexact}
    Assume that there is a t-structure $(\D_{\geq 0}, \D_{\leq 0})$ on $\D$ such that $F\colon \C\to\D$ is right t-exact and $G\colon\C\to\D$ is left t-exact. The forgetful functor $U$ is t-exact and there is a cartesian diagram
    \[\begin{tikzcd}
        \LEq(F^{\heart}, G^{\heart}) \ar{r}{U_{\C^{\heart}}} \ar[hook]{d} & \C^{\heart} \ar[hook]{d} \\
        \LEq(F,G) \ar{r}{U} & \C \rlap{.}
    \end{tikzcd}\]
\end{Corollary}

    \section{The left adjoint of the forgetful functor}\label{chap3}

Let $\C$ be an $\infty$-category that admits countable coproducts and let $F\colon \C \to \C$ be an endofunctor that preserves them. In this section we define the left adjoint of the forgetful functor $U\colon \Cart(\C,F) \to \C$. It follows that $U$ exhibits $\Cart(\C,F)$ as monadic over $\C$, cf.\ \cite[Definition 4.7.3.4]{HA}. This is the first key input for the proof of our main theorem. We conclude the section by showing that the left adjoint of $U$ is t-exact if $\C$ is a stable $\infty$-category equipped with a t-structure that is compatible with coproducts (cf. \cite[Definition 1.2.2.12]{HA}).

\begin{Definition}\label{DefL}
    Let $\C$ be an $\infty$-category that admits all countable coproducts and let $F\colon \C \to \C$ be an endofunctor that preserves countable coproducts. We define a functor $L\colon \C \to \Cart(\C,F)$ via the universal property of the pullback and the diagram
    \[\begin{tikzcd}
        \C \ar[dashed]{dr}{L} \ar[bend right]{ddr}[swap]{s} \ar[bend left]{drr}{\coprod_{n \geq 0}F^n} &  &  \\
         & \Cart(\C,F) \ar{r}{U} \ar{d}{\kappa} & \C \ar{d}{(F,\id)} \\
         & \Arr(\C) \ar{r}[swap]{(\ev_0,\ev_1)} & \C \times \C \rlap{,}
    \end{tikzcd}\]
    where $\coprod_{n \geq 0}F^n$ is the coproduct in $\Fun(\C,\C)$ of the functors $F^n\colon\C \to \C$ and the functor $s$ is defined as follows:
    
    Note that specifying a functor $\C \to \Arr(\C)$ is the same as giving an arrow in $\Fun(\C,\C)$, i.e.\ a natural transformation. Under this equivalence the functor $s$ corresponds to the natural transformation \[\coprod_{n \geq 0} \inc_{n+1}\colon \coprod_{n \geq 0} F^{n+1} \to \coprod_{k\geq 0} F^k\]
    where $\inc_{n+1}\colon F^{n+1} \to \coprod_{k\geq 0} F^k$ denotes the natural map into the coproduct.
    
    As $F$ preserves countable coproducts, the outer diagram commutes and hence defines a functor $L$.
\end{Definition}

\begin{Proposition}\label{Lexists}
    Let $\C$ be an $\infty$-category that admits all countable coproducts and let $F\colon \C \to \C$ be an endofunctor that preserves countable coproducts. Then the functor $L$ is left adjoint to the forgetful functor $U\colon \Cart(\C,F) \to \C$.
\end{Proposition}

\begin{proof}
    First, we define the unit map $u\colon \id_{\C} \to UL$ to be the inclusion of the zeroth component $\inc_0\colon \id_{\C}=F^0 \to \coprod_{n \geq 0} F^n$. By \cite[Proposition 5.2.2.8]{HTT} it is enough to show that the map
    \begin{align}\label{adjcomp1}
    \Map_{\Cart(\C,F)}(LX,M) \xrightarrow{U} \Map_{\C}(ULX,UM) \xrightarrow{u^*} \Map_{\C}(X,UM)
    \end{align}
    is an equivalence for all objects $X\in\C$ and $M\in\Cart(\C,F)$.
    
    Note that there are equivalences
    \begin{align*}
    \Map&_{\Cart(\C,F)}(LX,M) \simeq \Eq\left(\hspace{-5 pt}
        \begin{tikzcd}[ampersand replacement=\&]
            \Map_{\C}(ULX,UM) \ar[shift left=0.75ex]{r}{\kappa(LX)^*} \ar[shift right=0.75ex]{r}[swap]{\kappa(M)_*F}  \&[10 pt] \Map_{\C}(FULX,UM)
        \end{tikzcd}\hspace{-5 pt}\right) \\
        &\simeq \Eq\left(\hspace{-5 pt}\begin{tikzcd}[ampersand replacement=\&]
            \Map_{\C}\left(\displaystyle\coprod\limits_{n \geq 0}F^nX, UM\right)  \ar[shift left=0.75ex]{r} \ar[shift right=0.75ex]{r}  \& \Map_{\C}\left(\displaystyle\coprod\limits_{n \geq 0}F^{n+1}X, UM\right)
        \end{tikzcd}\hspace{-5 pt}\right) \\
        &\simeq \Eq\left(\hspace{-5 pt}\begin{tikzcd}[ampersand replacement=\&]
            \displaystyle\prod\limits_{n \geq 0} \Map_{\C}(F^nX,UM) \ar[shift left=0.75ex]{r}{\Phi} \ar[shift right=0.75ex]{r}[swap]{\Psi}  \& \displaystyle\prod\limits_{n \geq 0}\Map_{\C}(F^{n+1}X,UM)
        \end{tikzcd}\hspace{-5 pt}\right) \, ,
    \end{align*}
    where we used \Cref{properties}\ref{map} in the first equivalence, the definition of $L$ and the assumption that $F$ preserves countable coproducts in the second equivalence and the universal property of coproducts in the last equivalence. Using the definition of $L$ we see that the map $\Phi$ is defined by $\pr_n \Phi \coloneqq \pr_{n+1}$ where \[\pr_n\colon \prod_{k\geq 0} \Map_{\C}(F^kX,UM) \to \Map_{\C}(F^nX,UM)\] denotes the projection on the $n$-th factor. Moreover, the map $\Psi$ is defined by ${\pr_n\Psi \coloneqq \kappa(M)_*F\pr_n}$. Thus, the composition (\ref{adjcomp1}) is equivalent to the composition
    \begin{align}\label{adjcomp2}
        \Eq(\Phi, \Psi) \xrightarrow{i} \prod_{n \geq 0} \Map_{\C}(F^nX,UM) \xrightarrow{\pr_0} \Map_{\C}(X,UM) \, ,
    \end{align}
    where $i$ is the natural map out of the equalizer. 
    
    To show that this composition is an equivalence we construct a map \[\beta\colon \Map_{\C}(X,UM) \to \Eq(\Phi, \Psi)\] that satisfies $\pr_0 i \beta \simeq \id$ and $\beta \pr_0 i \simeq \id$ and hence is an inverse of (\ref{adjcomp2}). For that, we use the universal property of the equalizer and construct a map \[\alpha\colon \Map_{\C}(X,UM) \to \prod_{n \geq 0} \Map_{\C}(F^nX,UM)\] with a homotopy $\Phi \alpha \simeq \Psi \alpha$. We define inductively \[\pr_0 \alpha\coloneqq \id\colon \Map_{\C}(X,UM) \to \Map_{\C}(F^0X,UM)\] and 
    \[\begin{tikzcd}
        \Map_{\C}(X,UM) \ar{r}{\pr_{n+1}\alpha} \ar{d}[swap]{\pr_n\alpha} & \Map_{\C}(F^{n+1}X,UM) \\
        \Map_{\C}(F^nX,UM) \ar{r}[swap]{F} & \Map_{\C}(F^{n+1}X,FUM) \ar{u}[swap]{\kappa(M)_*} \rlap{.}
    \end{tikzcd}\]
    This induces a map $\alpha\colon \Map_{\C}(X,UM) \to \prod_{n \geq 0} \Map_{\C}(F^nX,UM)$ and
    \[ \Phi \alpha = (\pr_{n+1}\alpha)_n \simeq (\kappa(M)_*F\pr_n \alpha)_n = \Psi \alpha\]
    holds by construction. Therefore, we get a map $\beta\colon \Map_{\C}(X,UM) \to \Eq(\Phi, \Psi)$ that satisfies $i \beta \simeq \alpha$. 
    
    It follows immediately from the construction of $\alpha$ that we have \[\pr_0 i \beta \simeq \pr_0 \alpha \simeq \id \, .\] Thus, it remains to show that $\beta \pr_0 i \simeq \id$. For that, it suffices to show that for all $n$ we have $\pr_n i \beta \pr_0 i \simeq \pr_n i$. We do this inductively. For $n=0$ we have \[\pr_0 i \beta \pr_0 i \simeq \id \pr_0 i \simeq \pr_0 i \, \]
    by the above. For $n > 0$ we compute
    \begin{align*}
        \pr_{n+1} i \beta \pr_0 i &\simeq \pr_{n+1} \alpha \pr_0 i \\
        &\simeq \kappa(M)_*F\pr_n\alpha\pr_0 i \\
        &\simeq \kappa(M)_*F\pr_n i \\
        &\simeq \pr_{n+1} i
    \end{align*}
    where we used the definition of $\beta$ in the first equivalence, the definition of $\alpha$ in the second equivalence, the induction hypothesis in the third equivalence and the definition of $\Phi$ and $\Psi$ in the last equivalence. 
\end{proof}

\begin{Corollary}\label{leqLexists}
    Let $\C$ be an $\infty$-category that admits all countable coproducts and let $F\colon \C \to \D$ be a functor that preserves countable coproducts. Let $G\colon \C \to \D$ be a functor that admits a left adjoint $G^L\colon \D \to \C$. Then the forgetful functor $U\colon \LEq(F,G) \to \C$ admits a left adjoint. It can be described as in \Cref{DefL} by replacing $F$ by $G^L F$.
\end{Corollary}

\begin{proof}
    There is an equivalence of $\infty$-categories $\LEq(F,G) \simeq \Cart(\C,G^L F)$ by \Cref{leqadjoint}. Note that the functor $G^L$ preserves countable coproducts as it is a left adjoint. Hence, we can apply \Cref{Lexists}.
\end{proof}

\begin{Corollary}\label{frobLexists}
    Let $\C$ be an $\infty$-category that admits all countable coproducts and let $G\colon \C \to \C$ be an endofunctor that admits a left adjoint $G^L$. Then the forgetful functor $U\colon \Frob(\C,G) \to \C$ admits a left adjoint. It can be described as in \Cref{DefL} by replacing $F$ by $G^L$.
\end{Corollary}

\begin{proof}
    Apply \Cref{leqLexists} with $F=\id_{\C}$.
\end{proof}

As a consequence, we obtain that the forgetful functor is monadic. The proof of this is an application of the $\infty$-categorical Barr-Beck Theorem \cite[Theorem 4.7.3.5]{HA}. To see this, we first show the following.

\begin{Lemma}\label{Usplitcolimits}
    Let $\C$ and $\D$ be $\infty$-categories and $F,G\colon\C \to \D$ be functors. The $\infty$-categories $\Cart(\C,F)$, $\Frob(\C,G)$ (in the case $\D=\C$) and $\LEq(F,G)$ admit colimits of $U$-split simplicial objects (cf.\ \cite[Definition 4.7.2.2]{HA}) and $U$ preserves them.
\end{Lemma}

\begin{proof}
    Let $X_{\bullet}\colon N(\Delta_+^{\op})^{\op} \to \Cart(\C,F)$ be a $U$-split simplicial object, i.e.\ $UX_{\bullet}$ is a split simplicial object in $\C$. By \cite[Lemma 6.1.3.16]{HTT}, $UX_{\bullet}$ admits a colimit in $\C$. Note that the functor $F$ preserves this colimit by \cite[Remark 4.7.2.4]{HA}. Thus, the claim about $\Cart(\C,F)$ follows from \Cref{properties}\ref{colim}.
    
    To prove the claims about $\Frob(\C,G)$ and $\LEq(F,G)$ we can do an analogous argument using \Cref{frobproperties}\ref{frobcolim}, resp.\ \Cref{leqproperties}\ref{leqcolim} at the end.
\end{proof}

\begin{Theorem}\label{monadicity}
    \begin{enumerate}[label=(\alph*)]
        \item\label{Cartmonadicity} Let $\C$ be an $\infty$-category and let $F\colon \C \to \C$ be an endofunctor such that the forgetful functor $U\colon \Cart(\C,F) \to \C$ admits a left adjoint (e.g.\ if $\C$ admits countable coproducts and $F$ preserves them). Then the functor $U$ exhibits $\Cart(\C,F)$ as monadic over $\C$ (in the sense of \cite[Definition 4.7.3.4]{HA}).
        \item Let $\C$ and $\D$ be $\infty$-categories and let $F, G\colon \C \to \D$ be functors such that the forgetful functor $U\colon \LEq(F,G) \to \C$ admits a left adjoint. Then the functor $U$ exhibits $\LEq(F,G)$ as monadic over $\C$.
        \item Let $\C$ be an $\infty$-category and let $G\colon \C \to \C$ be an endofunctor such that the forgetful functor $U\colon \Frob(\C,G) \to \C$ admits a left adjoint. Then the functor $U$ exhibits $\Frob(\C,G)$ as monadic over $\C$.
    \end{enumerate}
\end{Theorem}

\begin{proof}
By the $\infty$-categorical Barr-Beck Theorem \cite[Theorem 4.7.3.5]{HA} it is enough to show that $U$ is conservative, $\Cart(\C,F)$ (resp.\ $\LEq(F,G)$ or $\Frob(\C,G)$) admits colimits of $U$-split simplicial objects and $U$ preserves these colimits. The latter holds by \Cref{Usplitcolimits}. Moreover, we have already seen that $U$ is conservative in \Cref{properties}\ref{Uconservative} (resp.\ \Cref{leqproperties}\ref{leqUconservative} or \Cref{frobproperties}\ref{frobUconservative}).
\end{proof}

We finish this section by investigating the t-exactness properties of the functor $L$. Under the additional assumption that the t-structure on $\C$ is compatible with coproducts we get an analog of \Cref{Utexact} which is an important ingredient for the proof of our main theorem.

\begin{Lemma}\label{Ltexact}
    Let $\C$ be a stable $\infty$-category with a t-structure $(\C_{\geq 0},\C_{\leq 0})$. Suppose that $\C_{\leq 0}$ is stable under coproducts.
    \begin{enumerate}[label=(\alph*)]
        \item\label{CartLtexact} Let $F\colon \C \to \C$ be a t-exact endofunctor that preserves countable coproducts. Then the left adjoint of the forgetful functor $U\colon \Cart(\C,F) \to \C$ from \Cref{DefL} is t-exact (with respect to the induced t-structure on $\Cart(\C,F)$ described in \Cref{tstructure}).
        \item\label{leqLtexact} Let $\D$ be a stable $\infty$-category with a t-structure and $F\colon \C \to \D$ be a t-exact functor that preserves countable coproducts. Let $G\colon \C \to \D$ be a functor that admits a t-exact left adjoint. Then the left adjoint of the forgetful functor $U\colon \LEq(F,G) \to \C$ is t-exact.
        \item\label{FrobLtexact} Let $G\colon \C \to \C$ be an endofunctor that admits a t-exact left adjoint. Then the left adjoint of the forgetful functor $U\colon \Frob(\C,G) \to \C$ is t-exact.
    \end{enumerate}
\end{Lemma}

\begin{proof}
    Note for \ref{leqLtexact} and \ref{FrobLtexact} that $G$ is left t-exact by \cite[Proposition 1.3.17(iii)]{BBD} because it is right adjoint to a t-exact functor. In particular, the $\infty$-categories $\LEq(F,G)$ and $\Frob(\C,G)$ carry the induced t-structure described in \Cref{tstructure}.
    
    In all three cases we denote the left adjoint of the forgetful functor by $L$. By \Cref{tstructure} we have to show that $UL(\C_{\geq 0}) \subseteq \C_{\geq 0}$ and that ${UL(\C_{\leq 0}) \subseteq \C_{\leq 0}}$. This follows directly from the definition of $L$ (\Cref{DefL}, \Cref{leqLexists}, \Cref{frobLexists}), the assumptions on $F$ and $G$ and the fact that $C_{\geq 0}$ and $\C_{\leq 0}$ are stable under coproducts.
\end{proof}

\begin{Corollary}\label{Lheart}
    Under the respective assumptions of the previous lemma there are commutative diagrams
    \begin{enumerate}[label=(\alph*)]
    \item\label{cartLheart} \[\begin{tikzcd}
        \C^{\heart} \ar{r}{L_{\C^{\heart}}} \ar[hook]{d} & \Cart(\C^{\heart},F^{\heart}) \ar[hook]{d} \\
        \C \ar{r}{L_{\C}} & \Cart(\C,F)
    \end{tikzcd}\]
    \item\label{leqLheart} \[\begin{tikzcd}
        \C^{\heart} \ar{r}{L_{\C^{\heart}}} \ar[hook]{d} & \LEq(F^{\heart},G^{\heart}) \ar[hook]{d} \\
        \C \ar{r}{L_{\C}} & \LEq(F,G)
    \end{tikzcd}\]
    \item\label{frobLheart} \[\begin{tikzcd}
        \C^{\heart} \ar{r}{L_{\C^{\heart}}} \ar[hook]{d} & \Frob(\C^{\heart},G^{\heart}) \ar[hook]{d} \\
        \C \ar{r}{L_{\C}} & \Frob(\C,G)
    \end{tikzcd}\]
    \end{enumerate}
    In each case the horizontal functors denote the left adjoint of the forgetful functor.
\end{Corollary}

\begin{proof}
    We only prove \ref{cartLheart}, an analogous proof shows \ref{leqLheart} and \ref{frobLheart}. For the sake of clarity, we add an index to $L$ and $U$ to indicate to which category they refer.
    
    As $L_{\C}$ is t-exact by \Cref{Ltexact}, it restricts to a functor \[L_{\C}^{\heart}\colon \C^{\heart} \to \Cart(\C,F)^{\heart}\, .\] The same is true for the forgetful functor $U_{\C}\colon \Cart(\C,F) \to \C$ by \Cref{Utexact}. The restricted functors $L_{\C}^{\heart}$ and $U_{\C}^{\heart}$ are adjoint by \cite[Proposition 1.3.17(iii)]{BBD} because $L_{\C}$ is left adjoint to $U_{\C}$. But again by \Cref{Utexact}, $U_{\C}^{\heart} \simeq U_{\C^{\heart}}$, so $L_{\C}^{\heart}$ and $L_{\C^{\heart}}$ are left adjoints of equivalent functors. Hence, $L_{\C}^{\heart}$ and $L_{\C^{\heart}}$ are equivalent (cf.\ \cite[Remark 5.2.2.2]{HTT}).
\end{proof}
    \section{The derived \texorpdfstring{$\infty$}{infinity}-category of Cartier modules}\label{chap4}

In this section we state and prove the main theorem of this paper. The proof relies on a consequence of the $\infty$-categorical Barr-Beck Theorem \cite[Theorem 4.7.3.5]{HA} which we recall in \Cref{barrbeck}.

Recall that a Grothendieck abelian category is a presentable abelian category that has exact filtered colimits. In particular, Grothendieck abelian categories satisfy Grothendieck's axioms AB3, AB4 and AB5, cf.\ \cite[\href{https://stacks.math.columbia.edu/tag/079A}{Tag 079A}]{Stacks}.

Furthermore, if $\A$ is a Grothendieck abelian category we can define its derived $\infty$-category $\D(\A)$, cf.\ \cite[Definition 1.3.5.8]{HA}. The derived $\infty$-category carries a natural t-structure (cf.\ \cite[Definition 1.3.5.16]{HA}) whose heart is given by $\D(\A)^{\heart} \simeq \A$ (cf.\ \cite[Remark C.5.4.11]{SAG}).

For $G\colon\A \to \B$ a colimit-preserving exact functor between Grothendieck abelian categories there is an induced colimit-preserving t-exact functor 
\[\D(G)\colon \D(\A) \to \D(\B)\]
by \Cref{UPderived}. Moreover, we can also apply \Cref{UPderived} to the equivalence 
\[\Arr(\A) \xrightarrow{\simeq} \Arr\left(\D(\A)^{\heart}\right) \xrightarrow{\simeq} (\Arr(\D(\A)))^{\heart}\]
to get a functor $\D(\Arr(\A)) \to \Arr(\D(\A))$.

\begin{Theorem}\label{derived}
    Let $\A$ be a Grothendieck abelian category and $F\colon \A \to \A$ be an exact and colimit-preserving functor.
    Then $\Cart(\A,F)$ is a Grothendieck abelian category and the natural functor $\alpha\colon \D(\Cart(\A,F)) \to \Cart(\D(\A),\D(F))$ given by the universal property of the pullback and the diagram
    \[\begin{tikzcd}
        \D(\Cart(\A,F)) \ar[bend left]{rrd}{\DU} \ar{dd}[swap]{\D(\kappa)} \ar[dashed]{dr}{\alpha} & & \\
        & \Cart(\D(\A),\D(F)) \ar{r}{\UD} \ar{d}[swap]{\kappa_{\D(\A)}} & \D(\A) \ar{d}{(\D(F),\id)} \\
        \D(\Arr(\A)) \ar{r} & \Arr(\D(\A)) \ar{r}[swap]{(\ev_0, \ev_1)} & \D(\A) \times \D(\A)
    \end{tikzcd}\]
    is an equivalence of $\infty$-categories.
    
    Furthermore, the functor $\alpha$ is t-exact with respect to the usual t-structure on $\D(\Cart(\A,F))$ and the induced t-structure on $\Cart(\D(\A),\D(F))$ described in \Cref{tstructure}.
\end{Theorem}

Note that the outer diagram in the theorem indeed commutes as this can be checked on hearts by \Cref{UPderived} where the bottom composition is given by $(\ev_0,\ev_1)$. So, it boils down to the commutative diagram defining $\Cart(\A,F)$.

\begin{Remark}
    By \Cref{UPderived}, the equivalence of the heart \[\Cart(\A,F) \simeq \Cart(\D(\A),\D(F))^{\heart}\] described in \Cref{heart} also induces a functor \[\D(\Cart(\A,F)) \to \Cart(\D(\A),\D(F)) \, .\]
    It can be shown that this functor is equivalent to the functor $\alpha$ of \Cref{derived} by checking it on the heart and applying \Cref{UPderived}.
\end{Remark}

To prove the theorem we use the following result which is a consequence of the $\infty$-categorical Barr-Beck Theorem \cite[Theorem 4.7.3.5]{HA}.

\begin{Proposition}\label{barrbeck}
    Suppose we are given a commutative diagram of $\infty$-categories
    \[\begin{tikzcd}
        \C \ar{rr}{\alpha} \ar{dr}[swap]{U} & &\C' \ar{dl}{U'} \\
        & \D &
    \end{tikzcd}\]
    and assume the following:
    \begin{enumerate}[label=(\arabic*)]
        \item\label{bb1} The functor $U$ exhibits $\C$ as monadic over $\D$.
        \item\label{bb2} The functor $U'$ exhibits $\C'$ as monadic over $\D$.
        \item\label{bb3} Denote the left adjoints of $U$ and $U'$ by $L$ and $L'$, respectively. For each object $D \in \D$, the unit map $D \to UL(D) \simeq U'\alpha L(D)$ induces an equivalence $L'(D) \to \alpha L(D)$ in $\C'$. 
    \end{enumerate}
    Then $\alpha$ is an equivalence.
\end{Proposition}

\begin{proof}
    This is a reformulation of \cite[Corollary 4.7.3.16]{HA} using \cite[Theorem 4.7.3.5]{HA}.
\end{proof}

So, to prove \Cref{derived} we show that for the diagram
\[\begin{tikzcd}
    \D(\Cart(\A,F)) \ar{rr}{\alpha} \ar{dr}[swap]{\DU} & & \Cart(\D(\A),\D(F)) \ar{dl}{\UD} \\
    & \D(\A) &
\end{tikzcd}\]
the assumptions \ref{bb1} to \ref{bb3} of \Cref{barrbeck} are satisfied where we keep the notation from \Cref{derived}. Note that the category $\Cart(\A,F)$ is Grothendieck abelian by \Cref{properties}\ref{Grothab}, so that it makes sense to write $\D(\Cart(\A,F))$. We start with the proof of \ref{bb1}.

\begin{Proposition}\label{DCartmonadic}
    Let $\A$ be a Grothendieck abelian category and $F\colon \A \to \A$ be an exact and colimit-preserving functor. Then the functor \[\D(U_{\A})\colon \D(\Cart(\A,F)) \to \D(\A)\] exhibits $\D(\Cart(\A,F))$ as monadic over $\D(\A)$.
\end{Proposition}

\begin{proof}
    By \Cref{Lexists} the functor $U_{\A}\colon \Cart(\A,F) \to \A$ admits a left adjoint $L_{\A}$. Using \Cref{monadicity}\ref{Cartmonadicity} we see that $U_{\A}$ exhibits $\Cart(\A,F)$ as monadic over $\A$. Noting that coproducts in Grothendieck abelian categories are exact by axiom AB4, we get that $L_{\A}$ is an exact functor. Thus, we conclude by applying \Cref{Dmonadic}.
\end{proof}

\vspace{10 pt}

\begin{proof}[Proof of \Cref{derived}]
    First note that the category $\Cart(\A,F)$ is Grothendieck abelian by \Cref{properties}\ref{Grothab}. 
    
    Moreover, the t-exactness of the functor $\alpha$ follows immediately from the definition of the induced t-structure on $\Cart(\D(\A),\D(F))$ and the t-exactness of $\DU \simeq \UD \alpha$. 
    
    So, it remains to show that the natural functor \[\alpha\colon \D(\Cart(\A,F)) \to \Cart(\D(\A),\D(F))\] is an equivalence. As we already described above the proof proceeds by considering the commutative diagram
    \[\begin{tikzcd}
    \D(\Cart(\A,F)) \ar{rr}{\alpha} \ar{dr}[swap]{\DU} & & \Cart(\D(\A),\D(F)) \ar{dl}{\UD} \\
    & \D(\A) &
    \end{tikzcd}\]
    and checking that the assumptions \ref{bb1} to \ref{bb3} of \Cref{barrbeck} are satisfied.

    \ref{bb1} is \Cref{DCartmonadic}.
    
    For \ref{bb2} note that $\D(\A)$ is presentable by \cite[Proposition 1.3.5.21(1)]{HA} and that $\D(F)$ preserves colimits by definition. Hence, we can apply \Cref{monadicity}\ref{Cartmonadicity} to get \ref{bb2}.

    It remains to prove \ref{bb3}. We have to show that for each object $M \in \D(\A)$ there is an equivalence $\LD(M) \to \alpha \DL(M)$ induced by the unit of the adjunction $\DL \dashv \DU$. As $\UD$ is conservative it suffices to show that the map \[\UD \LD(M) \to \UD \alpha \DL (M) \simeq \DU \DL(M)\] is an equivalence. 
    
    The functor $\UD$ is t-exact by \Cref{Utexact} and colimit-preserving by \Cref{properties}\ref{colimits}. Note that the t-structure on $\D(\A)$ is compatible with filtered colimits by \cite[Proposition 1.3.5.21(3)]{HA}, and hence, the functor $\LD$ is t-exact by \Cref{Ltexact}\ref{CartLtexact}. The functor $\LD$ also preserves colimits as it is a left adjoint. Moreover, the functors $\DU$ and $\DL$ are t-exact and colimit-preserving by definition. Thus, we can apply \Cref{UPderived} to reduce to the case $M\in \D(\A)^{\heart}$.
    
    Denote by $H\colon \D(\A)^{\heart} \to \D(\A)$ the inclusion of the heart. Then there is a commutative diagram
    \[\begin{tikzcd}
        \pi_0 \UD \LD H(M) \ar{r} \ar{d}[swap]{\simeq} & \pi_0 \DU \DL H(M) \ar{d}{\simeq} \\
        U_{\A} L_{\A} (M) \ar{r}{\id} & U_{\A} L_{\A} (M)
    \end{tikzcd}\]
    where the left vertical arrow is an equivalence because of \Cref{Lheart}\ref{cartLheart}, \Cref{Utexact} and the fact that $\pi_0 H \simeq \id$ and the right vertical arrow is an equivalence by definition of $\DU$ and $\DL$. The commutativity follows because the top horizontal map is induced by the composition
    \begin{align*}
    \UD\LD &\xrightarrow[\hspace{20 pt}]{\D(u)} \UD\LD\D(U_{\A})\D(L_{\A}) \\
    &\xrightarrow[\hspace{20 pt}]{\simeq} \UD\LD\UD \alpha \D(L_{\A}) \\
    &\xrightarrow[\hspace{20 pt}]{c_{\D(\A)}} \UD\alpha\D(L_{\A}) \\
    &\xrightarrow[\hspace{20 pt}]{\simeq} \DU\DL
    \end{align*}
    where $\D(u)$ denotes the unit of the adjunction $\DL \dashv \DU$ and $c_{\D(\A)}$ the counit of the adjunction $\LD \dashv \UD$. When restricted to the heart these two adjunctions both become equivalent to the adjunction $L_{\A} \dashv U_{\A}$ and the map comes down to \[U_{\A}L_{\A} \xrightarrow{u} U_{\A}L_{\A}U_{\A}L_{\A} \xrightarrow{c} U_{\A}L_{\A}\] which is equivalent to the identity because of one of the triangle identities.
    
    The diagram shows that the map \[\pi_0 \UD \LD H(M) \to \pi_0\DU\DL H(M)\] is an equivalence for $M \in \D(\A)^{\heart}$. As we explained above, this implies that the map \[\UD \LD (M) \to \DU\DL(M)\] is an equivalence for all $M\in \D(\A)$. This completes the proof of \ref{bb3} and thus also the proof of the theorem.
\end{proof}

\begin{Corollary}\label{leqmainthm}
    Let $\A$ be a Grothendieck abelian category and $F, G\colon \A \to \A$ be exact and colimit-preserving functors. Assume that $G$ has an exact left adjoint $G^L$. Then $\LEq(F,G)$ is a Grothendieck abelian category and the natural functor \[\alpha\colon \D(\LEq(F,G)) \to \LEq(\D(F),\D(G)) \, ,\] defined analogously to the one in \Cref{derived}, is a t-exact equivalence of $\infty$-categories.
\end{Corollary}

\begin{proof}
    Note that by \Cref{Dadjoint}, $\D(G^L)$ is left adjoint to $\D(G)$. So, by \Cref{leqadjoint} there are equivalences $\LEq(F,G) \simeq \Cart(\A,G^L F)$ and \[\LEq(\D(F),\D(G)) \simeq \Cart(\D(\A),\D(G^L) \D(F))\] of $\infty$-categories. Now we can apply \Cref{derived} to conclude.
\end{proof}

\begin{Corollary}\label{frobmainthm}
    Let $\A$ be a Grothendieck abelian category and $G\colon \A \to \A$ be an exact and colimit-preserving functor. Assume that $G$ has an exact left adjoint. Then $\Frob(\A,G)$ is a Grothendieck abelian category and the natural functor \[\alpha\colon \D(\Frob(\A,G)) \to \Frob(\D(\A),\D(G)) \, ,\] defined analogously to the one in \Cref{derived}, is a t-exact equivalence of $\infty$-categories.
\end{Corollary}

\begin{proof}
    Apply \Cref{leqmainthm} with $F = \id_{\A}$.
\end{proof}

If we apply \Cref{derived} with $F=\id_{\A}$ we recover the following well-known result about the category $\End(\A)$ of endomorphisms of $\A$.

\begin{Corollary}\label{endmainthm}
    Let $\A$ be a Grothendieck abelian category. Then there is a canonical equivalence $\D(\End(\A)) \xrightarrow{\simeq} \End(\D(\A))$ of $\infty$-categories.
\end{Corollary}
    \section{Applications}\label{chap5}

In this section we collect some applications of our main theorem. These include in particular the corresponding statements about (classical) Cartier modules on a scheme over a field of positive characteristic. We also construct a perverse t-structure on classical Cartier modules. As the perverse t-structure is usually denoted by cohomological notation in the literature, we chose to do the same here, although we used homological notation before.

\begin{Notation}
    For a scheme $X$, let $\Mod(X)$ 
    be the Grothendieck abelian category of $\O_X$-modules, and $\QCoh(X)$ be the subcategory of quasi-coherent $\O_X$-modules,
    which itself is Grothendieck abelian. 
    
    If $p>0$ is a prime and $X$ is an $\FF_p$-scheme, we denote by $F\colon X\to X$ the absolute Frobenius of $X$.
    Consider the associated pushforward functor $F_* \colon \Mod(X) \to \Mod(X)$.
    Since $F$ is affine, it is separated and quasi-compact \cite[\href{https://stacks.math.columbia.edu/tag/01S7}{Tag 01S7}]{Stacks},
    and hence it restricts to a functor $F_* \colon \QCoh(X) \to \QCoh(X)$ \cite[\href{https://stacks.math.columbia.edu/tag/01LC}{Tag 01LC}]{Stacks}.
    On both categories the functor $F_*$ is exact since $F$ is topologically the identity and thus does 
    nothing on stalks (viewed as abelian groups). 
    Hence, $F$ induces derived functors $\D(F_*) \colon \D(\Mod(X)) \to \D(\Mod(X))$ 
    and $\D(F_*) \colon \D(\QCoh(X)) \to \D(\QCoh(X))$ which are t-exact with respect to the standard t-structures. 
\end{Notation}

\begin{Corollary}\label{classicalCart}
    Let $p > 0$ be a prime and $X$ be an $\FF_p$-scheme. Then there are canonical t-exact equivalences of $\infty$-categories 
    \[\D(\Cart(\QCoh(X), F_*)) \xrightarrow{\simeq} \Cart(\D(\QCoh(X)),\D(F_*))\] 
    and 
    \[\D(\Cart(\Mod(X), F_*)) \xrightarrow{\simeq} \Cart(\D(\Mod(X)),\D(F_*)).\]
\end{Corollary}

\begin{proof}
    Apply \Cref{derived} with the Grothendieck abelian category $\QCoh(X)$ and the exact and colimit-preserving functor $F_*\colon \QCoh(X) \to \QCoh(X)$
    for the first equivalence.
    For the other replace $\QCoh(X)$ by $\Mod(X)$.
\end{proof}

\begin{Corollary}\label{classicalFrob}
     Let $p>0$ be a prime and $X$ be a regular Noetherian $\FF_p$-scheme. 
     Then there are canonical t-exact equivalences of $\infty$-categories 
    \[\D(\Frob(\QCoh(X), F_*)) \xrightarrow{\simeq} \Frob(\D(\QCoh(X)),\D(F_*))\] 
    and 
    \[\D(\Frob(\Mod(X), F_*)) \xrightarrow{\simeq} \Frob(\D(\Mod(X)),\D(F_*)).\]
\end{Corollary}

\begin{proof}
    First note that the functor $F^*\colon \QCoh(X) \to \QCoh(X)$ is left adjoint to 
    $F_*$ and exact by \cite[Theorem 2.1]{Kunz}. 
    Thus, for the first equivalence we can apply \Cref{frobmainthm} to $\A = \QCoh(X)$ and $G=F_*$.
    For the second equivalence replace $\QCoh(X)$ by $\Mod(X)$.
\end{proof}

Another application of \Cref{derived} is that it enables us to define t-structures on $\D(\Cart(\A,F))$ by giving an induced t-structure on $\Cart(\D(\A),\D(F))$ in the sense of \Cref{chap2}. 

As an example we construct a perverse t-structure on the bounded derived category of classical Cartier modules with coherent cohomology. Perverse t-structures in general were first studied in \cite{BBD} and later more generally by Gabber \cite{Gabber}. As they are usually denoted by cohomological notation in the literature, we chose to do the same here, although we used homological notation before.

\begin{Notation}
    If $\D$ is a stable $\infty$-category with a t-structure $(\D^{\le 0}, \D^{\ge 0})$,
    we write 
    \[\D^b \coloneqq \bigcup_{\substack{a \to \infty\\b\to -\infty}} \D^{\leq a} \cap \D^{\geq b} \]
    for the subcategory of bounded objects.
    Similarly, denote by 
    \[\D^+ \coloneqq \bigcup_{b \to -\infty} \D^{\geq b} \]
    the full subcategory of those objects that are bounded below.
\end{Notation}

We introduce some more notation. 
For a scheme $X$ and a point $x \in X$ we denote by $i_x \colon \Spec(\O_{X,x}) \to X$
the inclusion of the local scheme at $x$. Recall that there is an exact inverse image functor 
\[i_x^*\colon \Mod(X) \to \Mod(\O_{X,x})\]
and a left exact functor 
\[i_x^!\colon \Mod(X) \to \Mod(\O_{X,x})\]
given by taking sections with support in $\overline{\{x\}}$ 
(compare with \cite[Section IV.1, Variation 8]{RD}). 
We denote their derived, resp.\ right derived functors by 
\[i_x^*\colon \D(\Mod(X)) \to \D(\Mod(\O_{X,x}))\,,\]
resp.\ 
\[Ri_x^!\colon \D^+(\Mod(X)) \to \D^+(\Mod(\O_{X,x})) \, .\]

\begin{Remark}
    \emph{A priori} it is not clear that the right derived functor $Ri_x^!$ exists on the $\infty$-categorical level.
    We will not need this in the sequel, as we only need the value of $Ri_x^!$ on objects.
    Nonetheless, one can construct $Ri_x^!$ on the $\infty$-categorical level via the adjoint functor theorem,
    since $\D(i_{x,*})$ preserves all colimits by construction,
    and all involved derived $\infty$-categories are presentable.
\end{Remark}

\begin{Theorem}[{\cite[§6]{Gabber}}]
    Let $X$ be a Noetherian scheme. 
    Let $d\colon X \to \ZZ \cup \{\infty\}$ be a bounded below weak perversity function in the sense of \cite[§1]{Gabber}. 
    Define full subcategories ${}^d \D^{\geq 0}(\Mod(X))$ and ${}^d \D^{\leq 0}(\Mod(X))$ of $\D(\Mod(X))$ as follows:
    Let $M \in \D(\Mod(X))$ be an object. Then
    \begin{itemize}
        \item $M \in {}^d \D^{\leq 0}(\Mod(X))$ if we have $i_x^*(M) \in \D^{\leq d(x)}(\Mod(\O_{X,x}))$ for every point $x \in X$, and
        \item $M \in {}^d \D^{\geq 0}(\Mod(X))$ if $M \in \D^+(\Mod(X))$ and we have $Ri_x^!(M) \in \D^{\geq d(x)}(\Mod(\O_{X,x}))$ for every point $x \in X$.
    \end{itemize}
    These subcategories define a t-structure on $\D(\Mod(X))$, 
    which is called the \emph{perverse t-structure} (with respect to $d$).
\end{Theorem}


Observe that the standard t-structure is a special case of a perverse t-structure.

\begin{Lemma}\label{perversityzero}
    Let $X$ be a Noetherian scheme. 
    The perverse t-structure with respect to $d = 0$ 
    is equivalent to the standard t-structure on $\D(\Mod(X))$.
\end{Lemma}

\begin{proof}
    As t-structures are uniquely determined by their connective part it suffices to show 
    ${}^0 \D^{\leq 0}(\Mod(X)) = \D^{\leq 0}(\Mod(X))$.
    The lemma thus follows from the fact that the stalk functors $i_x^*$ are jointly conservative and exact 
    (and hence commute with cohomology).
\end{proof}

We now want to define the induced perverse t-structure on classical Cartier modules. 
To be able to apply the results from \Cref{chap2} we first have to show that 
the Frobenius pushforward functor is right t-exact. 
We more generally show the following:

\begin{Lemma}\label{F_*perversetexact}
    Let $X$ be a Noetherian scheme,
    and let $G \colon X \to X$ be an endomorphism that is 
    topologically the identity 
    (e.g.\ $X$ is a positive characteristic scheme, and $G = F$ is the Frobenius endomorphism).
    Then $G_* \colon \D(\Mod(X)) \to \D(\Mod(X))$ is right t-exact with respect to 
    the perverse t-structure with respect to any bounded below weak perversity function $d \colon X \to \ZZ \cup \{\infty\}$.
\end{Lemma}

\begin{proof}
    Since $G$ is topologically the identity, 
    for every $M \in \Mod(X)$ we have $(G_*M)(U) = M(G^{-1} U) = M(U)$ (as abelian groups),
    and in particular, for every $x \in X$ it follows that $(G_*M)_x = M_x$.
    In particular, $G_* \colon \Mod(X) \to \Mod(X)$ is exact. 
    Thus, the derived functor $\D(G_*) \colon \D(\Mod(X)) \to \D(\Mod(X))$
    exists and is t-exact for the standard t-structures.

    Let $M \in {}^d \D^{\leq 0}(\Mod(X))$, i.e.\ we have that 
    $M_x \cong i^*_x M \in \D^{\leq d(x)}(\Mod(\O_{X,x}))$ for every point $x \in X$.
    We have to see that the same is true for $(\D(G_*) (M))_x \cong i_x^* \D(G_*) (M)$.
    This is clear since $(\D(G_*)(M))_x \cong M_x$, as this can be checked on cohomology.
\end{proof}

\begin{Theorem}\label{perversethm}
    Let $p > 0$ be a prime and $X$ be a Noetherian $\FF_p$-scheme. 
    Then the perverse t-structure with respect to any bounded below weak perversity function $d \colon X \to \ZZ \cup \{ \infty \}$ 
    induces a perverse t-structure on $\D(\Cart(\Mod(X), F_*))$ such that the forgetful functor 
    $\D(U)\colon \D(\Cart(\Mod(X), F_*)) \to \D(\Mod(X))$ is t-exact for the perverse t-structures.
\end{Theorem}

\begin{proof}
    We showed in \Cref{classicalCart} that there is an equivalence of $\infty$-categories
    \[ \alpha\colon \D(\Cart(\Mod(X), F_*)) \xrightarrow{\simeq} \Cart(\D(\Mod(X)),\D(F_*))\]
    which is t-exact for the standard t-structure on $\D(\Cart(\Mod(X), F_*))$ and the induced t-structure on $\Cart(\D(\Mod(X)),\D(F_*))$. 
    Using this, it is left to show that there is an induced perverse t-structure on 
    $\Cart(\D(\Mod(X)),\D(F_*))$. 
    This follows directly from \Cref{F_*perversetexact} and \Cref{tstructure}.
    The t-exactness of the forgetful functor $\D(U)\colon \D(\Cart(\Mod(X), F_*)) \to \D(\Mod(X))$ follows from the definition of $\alpha$ because $\D(U)$ can be written as the composition
    \[\D(\Cart(\Mod(X), F_*)) \xrightarrow{\alpha} \Cart(\D(\Mod(X)),\D(F_*)) \xrightarrow{U} \D(\Mod(X))\]
    of two t-exact functors.
\end{proof}

Under suitable conditions, Gabber shows that the perverse t-structure on $\D(\Mod(X))$ 
restricts to a t-structure on derived modules with quasi-coherent (resp.\ coherent) cohomology.
We now show that this also works for derived Cartier modules. For this, we need the following definitions.

\begin{Definition}
    Let $p>0$ be a prime and $X$ be an $\FF_p$-scheme. 
    A (classical) Cartier module $M \in \Cart(\Mod(X), F_*)$ is called 
    \emph{quasi-coherent} (resp.\ \emph{coherent}) if its underlying $\O_X$-module 
    $UM$ is quasi-coherent (resp.\ coherent).
\end{Definition}

\begin{Notation}
    Let $p>0$ be a prime and $X$ be an $\FF_p$-scheme. 
    We write 
    \[ \D_{\mathrm{coh}}(\Cart(\Mod(X), F_*)) \subset \D_{\mathrm{qcoh}}(\Cart(\Mod(X), F_*)) \subset \D(\Cart(\Mod(X), F_*)) \]
    for the full subcategories consisting of those derived Cartier modules 
    that have coherent or quasi-coherent cohomology, respectively.

    Similarly, we define $\D_{\mathrm{coh}}(\Mod(X)) \subset \D_{\mathrm{qcoh}}(\Mod(X)) \subset \D(\Mod(X))$.
\end{Notation}

\begin{Lemma} \label{qcoh-cart}
    Let $p > 0$ be a prime and $X$ be a Noetherian $\FF_p$-scheme. 
    The equivalence $\alpha \colon \D(\Cart(\Mod(X), F_*)) \to \Cart(\D(\Mod(X)), \D(F_*))$
    from \Cref{classicalCart} restricts to an equivalence 
    \[ \D_{\mathrm{qcoh}}(\Cart(\Mod(X), F_*)) \xrightarrow{\simeq} \Cart(\D_{\mathrm{qcoh}}(\Mod(X)), \D(F_*)). \]
    If moreover $F$ is finite, then this further restricts to an equivalence 
    \[ \D_{\mathrm{coh}}(\Cart(\Mod(X), F_*)) \xrightarrow{\simeq} \Cart(\D_{\mathrm{coh}}(\Mod(X)), \D(F_*)). \]
\end{Lemma}
\begin{proof}
    Since $F$ is affine, it is in particular separated and quasi-compact \cite[\href{https://stacks.math.columbia.edu/tag/01S7}{Tag 01S7}]{Stacks}, hence 
    $\D(F_*)$ preserves quasi-coherent cohomology \cite[\href{https://stacks.math.columbia.edu/tag/01LC}{Tag 01LC}]{Stacks}.
    If $F$ is finite (hence proper), then $\D(F_*)$ moreover preserves coherent cohomology \cite[\href{https://stacks.math.columbia.edu/tag/02O5}{Tag 02O5}]{Stacks}.
    Hence, the $\infty$-categories on the right-hand side are defined.

    Since limits of fully faithful functors are fully faithful, we indeed have that 
    $\Cart(\D_{\mathrm{qcoh}}(\Mod(X)), \D(F_*)) \to \Cart(\D(\Mod(X)), \D(F_*))$ is again fully faithful,
    and similarly for the coherent version.

    It thus suffices to compare the respective subcategories on both sides of the equivalence $\alpha$.
    Since $\alpha$ commutes with the forgetful functors, and since the forgetful functors are t-exact,
    they agree.
\end{proof}

\begin{Proposition}
    Let $p > 0$ be a prime, $X$ be a Noetherian $\FF_p$-scheme and $d \colon X \to \ZZ \cup \{ \infty \}$ a bounded below weak perversity function.
    Consider the perverse t-structure on $\D(\Cart(\Mod(X), F_*))$.
    \begin{enumerate}[label=(\alph*)]
        \item If $d$ is a strong perversity function \cite[§1]{Gabber}, then this t-structure restricts to $\D_{\mathrm{qcoh}}(\Cart(\Mod(X), F_*))$.
        \item If $F$ is finite and $d$ is a finite-valued strong perversity function (i.e.\ lands in $\ZZ \subset \ZZ \cup \{ \infty \}$) satisfying \cite[condition (c) from §9]{Gabber}, then this t-structure restricts to $\D_{\mathrm{coh}}(\Cart(\Mod(X), F_*))$.
    \end{enumerate}
\end{Proposition}
\begin{proof}
    If $d$ is a strong perversity function, then the perverse t-structure on $\D(\Mod(X))$ restricts to $\D_{\mathrm{qcoh}}(\Mod(X))$ 
    by \cite[Remarks 6.1 (6)]{Gabber}. Hence, in this case the proposition immediately follows from \Cref{qcoh-cart,F_*perversetexact,tstructure}.

    For the second part, since $X$ is Noetherian and $F$ is finite, $X$ admits a dualizing complex \cite[Remark 13.6]{Gabber}.
    Hence, if $d$ is moreover finite-valued and strong, then the perverse t-structure on $\D(\Mod(X))$ restricts to $\D_{\mathrm{coh}}(\Mod(X))$ 
    by \cite[Theorem 9.1]{Gabber}. Hence, the same proof as above works.
\end{proof}

Let $X$ be a Noetherian scheme that admits a dualizing complex $\omega_X^{\bullet}$.
By \cite[Section V.7]{RD} the complex $Ri_x^!\omega_X^{\bullet}$ is concentrated in a single homological degree $d(x)$ for each point $x\in X$. 
The obtained map $d\colon X \to \ZZ$ is a finite-valued strong perversity function satisfying \cite[condition (c) from §9]{Gabber}.
Let us write $\operatorname{Coh}(X) \subset \QCoh(X)$ for the full subcategory of coherent modules.
Assume that $F$ is finite.
Using duality theory, Baudin defined in \cite[Definition 5.2.1]{Baudin} a perverse t-structure on
$\D^b(\Cart(\operatorname{Coh}(X), F_*))$ for the perversity function $d$ described above.
We can compare this t-structure with the perverse t-structure on $\D_{\mathrm{coh}}(\Cart(\Mod(X), F_*))$.

\begin{Corollary} \label{baudin}
    Let $p>0$ be a prime and $X$ be a Noetherian $\FF_p$-scheme that admits a dualizing complex.
    Assume moreover that the absolute Frobenius $F$ of $X$ is finite.
    The canonical functor 
    \begin{equation*}
        \D^b(\Cart(\operatorname{Coh}(X), F_*)) \to \D_{\mathrm{coh}}(\Cart(\Mod(X), F_*))
    \end{equation*}
    is perverse t-exact.
\end{Corollary}

\begin{proof}
    This is clear as this can be checked after applying the perverse t-exact conservative forgetful functor $\D(U)$.
\end{proof}

\begin{Remark}
    In the last corollary, if one furthermore assumes that $X$ has affine diagonal,
    one can show that the canonical functor 
    \begin{equation*}
        \D^b(\Cart(\operatorname{Coh}(X), F_*)) \to \D^b_{\mathrm{coh}}(\Cart(\Mod(X), F_*))
    \end{equation*}
    is an equivalence, which is t-exact for the perverse t-structures.
\end{Remark}
    \appendix
\section{Derived \texorpdfstring{$\infty$}{infinity}-categories}\label{appendix}

In this appendix we state and prove some facts about derived $\infty$-categories of Grothendieck abelian categories that we use to prove our main theorem.

Recall that a t-structure $(\C_{\geq 0},\C_{\leq 0})$ on an $\infty$-category $\C$ is called right complete if the natural functor 
\[\C \xrightarrow{\simeq} \lim(\dots \to \C_{\geq 0} \xrightarrow{\taug{1}} \C_{\geq 1} \xrightarrow{\taug{2}} \C_{\geq 2} \to \dots)\]
is an equivalence of $\infty$-categories. It is called right separated if $\bigcap_{n} \C_{\leq n} = 0$ and left separated if $\bigcap_{n} \C_{\geq n} = 0$.

\begin{Notation}
    Let $\A$ be a Grothendieck abelian category. We denote by $\D(\A)$ the derived $\infty$-category of $\A$, cf.\ \cite[Definition 1.3.5.8]{HA}. 
    
    It carries a t-structure $(\D(\A)_{\geq 0}, \D(\A)_{\leq 0})$ (cf.\ \cite[Definition 1.3.5.16]{HA}) which is accessible, right complete and compatible with filtered colimits by \cite[Proposition 1.3.5.21]{HA}. 
    
    For each $n$ we denote by $\iotag{n}\colon \D(\A)_{\geq n} \leftrightarrows \D(\A)\colon \taug{n}$ the adjunction of the inclusion and the connective cover functor and by $\taul{n}\colon \D(\A) \leftrightarrows \D(\A)_{\leq n}\colon \iotal{n}$ the adjunction of the truncation functor and the inclusion. 
    
    Furthermore, the heart of the t-structure is given by $\D(\A)^{\heart} \simeq \A$ (cf.\ \cite[Remark C.5.4.11]{SAG}) and we denote the inclusion of the heart by $H\colon \A \to \D(\A)$.
\end{Notation}

The next proposition can be seen as a universal property of derived $\infty$-categories. It shows that t-exact and colimit-preserving functors out of a derived $\infty$-category are uniquely determined by their restrictions to the heart.

\begin{Proposition}\label{UPderived}
    Let $\A$ be a Grothendieck abelian category and $\C$ a presentable stable $\infty$-category. Assume that $\C$ carries a right complete and left separated t-structure $(\C_{\geq 0},\C_{\leq 0})$ that is compatible with filtered colimits. Then the restriction functor \[\LFun^{\tex} ( \D(\A), \C) \to \LFun^{\ex}(\A, \C^{\heart}), \quad G \mapsto \pi_0 G H\] is an equivalence, where objects in the source are colimit-preserving t-exact functors from $\D(\A)$ to $\C$ and objects in the target are colimit-preserving exact functors from $\A$ to $\C^{\heart}$.
\end{Proposition}

\begin{proof}
    By \cite[Theorem C.5.4.9]{SAG} the restriction to the heart induces an equivalence 
    \[\LFun^{\lex}(\D(\A)_{\geq 0}, \C_{\geq 0}) \xrightarrow{\simeq} \LFun^{\ex}(\A,\C^{\heart})\]
    where objects in the source are colimit-preserving and left exact functors from $\D(\A)_{\geq 0}$ to $\C_{\geq 0}$. Therefore, it suffices to show that the restriction functor 
    \[\LFun^{\tex} ( \D(\A), \C) \to \LFun^{\lex}(\D(\A)_{\geq 0}, \C_{\geq 0})\]
    is an equivalence. 
    
    By \cite[Proposition 1.3.5.21(3)]{HA} the t-structure on $\D(\A)$ is right complete, i.e.\ $\D(\A) \xrightarrow{\simeq} \lim_{n} \D(\A)_{\geq -n}$. Thus, combining \cite[Remark C.1.2.10]{SAG} and \cite[Proposition C.3.1.1]{SAG} we get an equivalence 
    \[\LFun^{\rtex} ( \D(\A), \C) \xrightarrow{\simeq} \LFun(\D(\A)_{\geq 0}, \C_{\geq 0})\] 
    where objects in the source are colimit-preserving right t-exact functors from $\D(\A)$ to $\C$ and objects in the target are colimit-preserving functors from $\D(\A)_{\geq 0}$ to $\C_{\geq 0}$. But by \cite[Proposition C.3.2.1]{SAG} this induces an equivalence 
    \[\LFun^{\tex} ( \D(\A), \C) \xrightarrow{\simeq} \LFun^{\lex}(\D(\A)_{\geq 0}, \C_{\geq 0}) \, .\]
\end{proof}

In particular, by \cite[Remark C.5.4.11]{SAG} the above lemma applies if we take $\C$ to be the derived $\infty$-category $\D(\B)$ of a Grothendieck abelian category $\B$. 

\begin{Notation}
    Let $\A$ and $\C$ be as in \Cref{UPderived}. We denote by
    \[\D\colon \LFun^{\ex}(\A,\C^{\heart}) \to \LFun^{\tex}(\D(\A),\C), \quad G \mapsto \D(G)\]
    the inverse to the restriction functor.
\end{Notation}

In the following we investigate the properties of the functor $\D(G)$ if $G\colon \A \to \C^{\heart}$ is an exact and colimit-preserving functor. In the case where $\C$ is also a derived $\infty$-category, we show that $\D$ preserves adjoint functors and monadicity. To that end we start by showing that $\D$ preserves conservativity of functors. For that we need the following lemma.

\begin{Lemma}\label{Dpi}
    Let $\A$ and $\C$ be as in \Cref{UPderived} and $G\colon\A\to\C^{\heart}$ an exact and colimit-preserving functor. Then there is an equivalence $\pi_n \D(G) \simeq G \pi_n$ of functors $\D(\A) \to \C^{\heart}$ for each $n$.
\end{Lemma}

\begin{proof}
    We show the claim for $n=0$. Then the lemma follows from \[\pi_n \D(G) \simeq \pi_0 \Sigma^{-n} \D(G) \simeq \pi_0 \D(G) \Sigma^{-n} \simeq G \pi_0 \Sigma^{-n} \simeq G \pi_n\] where we used that $\D(G)$ is exact in the second equivalence. \\
    Note that by the definition of the restriction functor of \Cref{UPderived} we have $G \simeq \pi_0 \D(G) H$. 
    We show that there are equivalences \begin{equation}\label{Dpimap} \pi_0 \D(G) H \pi_0 \xrightarrow{\simeq} \pi_0 \D(G) \taul{0} \xleftarrow{\simeq} \pi_0 \D(G)\end{equation}
    where the first map is given by the counit of the adjunction $\iotag{0} \dashv \taug{0}$ (note that $H\pi_0 \simeq \taug{0}\taul{0}$) and the second map is given by the unit of the adjunction $\taul{0} \dashv \iotal{0}$. 
    
    To see that the first map is an equivalence consider the fiber sequence 
    \[H\pi_0 X = \taug{0}X \to X \to \taul{-1}X\]
    for $X \in \D(\A)_{\leq 0}$. Applying $\D(G)$ we get another fiber sequence 
    \[\D(G)H\pi_0 X \to \D(G)X \to \D(G)\taul{-1}X\]
    where the last term is $(-1)$-truncated and the other terms are $0$-truncated because of the t-exactness of $\D(G)$. This shows that the induced map 
    \[\pi_0\D(G)H\pi_0 X \to \pi_0\D(G)X\] 
    is an equivalence. Hence, also the first map of (\ref{Dpimap}) is an equivalence.
    
    To check that the second map of (\ref{Dpimap}) is an equivalence we consider the fiber sequence 
    \[\taug{1}Y \to Y \to \taul{0}Y\]
    for $Y\in\D(\A)$. Applying $\D(G)$ we get another fiber sequence whose first term is $1$-connective and whose last term is $0$-truncated. Hence, the unit of the adjunction $\taul{0} \dashv \iotal{0}$ induces an equivalence $\pi_0 \D(G) Y \simeq \pi_0 \D(G) \taul{0} Y$. So, also the second map of (\ref{Dpimap}) is an equivalence, which completes the proof.
\end{proof}

Using the above lemma we can conclude that $\D$ preserves conservativity of functors. Note that if an $\infty$-category $\C$ carries a right and left separated t-structure then a morphism $f$ in $\C$ is an equivalence if and only if for every $n$, the induced map $\pi_n(f)$ is an equivalence.

\begin{Corollary}\label{Dcons}
    Let $\A$ and $\C$ be as in \Cref{UPderived} and $G\colon\A\to\C^{\heart}$ an exact and colimit-preserving functor that is conservative. Then the induced functor $\D(G)$ is conservative.
\end{Corollary}

\begin{proof}
    Let $f$ be a morphism in $\D(\A)$ such that $\D(G)(f)$ is an equivalence. It follows that for each $n\in\ZZ$ the morphism $\pi_n\D(G)(f)$ is an equivalence. But by \Cref{Dpi} this morphism is equivalent to the morphism $G\pi_n(f)$. Using that $G$ is conservative we get that $\pi_n(f)$ is an equivalence for every $n$. Note that the t-structure on $\D(\A)$ is right complete by \cite[Proposition 1.3.5.21(3)]{HA}, so in particular right separated by the dual of \cite[Proposition 1.2.1.19]{HA}. Furthermore, it is left separated by \cite[Remark C.5.4.11]{SAG}. Thus, $\pi_n(f)$ being an equivalence already implies that $f$ is an equivalence.
\end{proof}

We finish this section by showing that $\D$ preserves adjoints and monadicity.

\begin{Lemma}\label{Dadjoint}
    Let $\A$ and $\B$ be Grothendieck abelian categories and $L\colon\A\to\B$ and $R\colon\B\to\A$ be exact and colimit-preserving functors such that $L$ is left adjoint to $R$. Then $\D(L)$ is left adjoint to $\D(R)$.
\end{Lemma}

\begin{proof}
    We define the unit and counit as
    \[ u_{\D}\colon \id_{\D(\A)} \simeq \D(\id_{\A}) \xrightarrow{\D(u)} \D(RL) \simeq \D(R)\D(L)\]
    and
    \[c_{\D}\colon \D(L)\D(R) \simeq \D(LR) \xrightarrow{\D(c)} \D(\id_{\B}) \simeq \id_{\D(\B)} \, ,\]
    respectively, where $u\colon \id_{\A} \to RL$ and $c\colon LR \to \id_{\B}$ denote the unit and counit of the adjunction $L\dashv R$. To check that they satisfy the triangle identities (which is enough by \cite[Remark 4.4.5]{RV}) consider the following diagram whose commutativity can be checked on the heart by using the equivalence from \Cref{UPderived}
    \[\begin{tikzcd}[column sep=small]
        \D(L) \ar{rrr}{u_{\D}} \ar{d}[swap]{\simeq} &&&[-25pt] \D(L)\D(R)\D(L) \ar{rrr}{c_{\D}} \ar{dl}{\simeq} \ar{dr}[swap]{\simeq} &[-25pt]&& \D(L) \ar{d}{\simeq} \\
        \D(L)\D(\id_{\A}) \ar{rr}{\D(u)} \ar{d}[swap]{\simeq} && \D(L)\D(RL) \ar{dr}{\simeq} && \D(LR)\D(L) \ar{rr}{\D(c)} \ar{dl}[swap]{\simeq} && \D(\id_{\B})\D(L) \ar{d}{\simeq} \\
        \D(L\id_{\A}) \ar{rrr}[swap]{\D(Lu)} &&& \D(LRL) \ar{rrr}[swap]{\D(cL)} &&& \D(\id_{\B}L) \rlap{.}
    \end{tikzcd}\]
    Note that the lower horizontal composition is equivalent to $\D(cL \circ Lu)$ which is equivalent to the identity by a triangle identity for $u$ and $c$. Thus, the top horizontal composition in the diagram is equivalent to the identity. The other triangle identity can be shown analogously.
\end{proof}

\begin{Remark}
    One can also extend the functor $\D$ to a functor from the $(2,2)$-category of Grothendieck abelian categories with colimit-preserving exact functors to the $(\infty,2)$-category of stable $\infty$-categories with t-structure with colimit-preserving t-exact functors. From this description the result of \Cref{Dadjoint} is a formal consequence.
\end{Remark}

\begin{Corollary}\label{Dmonadic}
    Let $\A$ and $\B$ be Grothendieck abelian categories and $L\colon\A\to\B$ and $R\colon\B\to\A$ be exact and colimit-preserving functors such that $L$ is left adjoint to $R$. Assume that $R$ exhibits $\B$ as monadic over $\A$. Then $\D(R)$ exhibits $\D(\B)$ as monadic over $\D(\A)$.
\end{Corollary}

\begin{proof}
    Note that by \Cref{Dadjoint} the functor $\D(R)$ has a left adjoint. By the $\infty$-categorical Barr-Beck Theorem (cf. \cite[Theorem 4.7.3.5]{HA}) it is enough to show that $\D(R)$ is conservative, $\D(\B)$ admits colimits (of $\D(R)$-split simplicial objects) and that $\D(R)$ preserves these colimits. Using the $\infty$-categorical Barr-Beck Theorem for $R$ we get that $R$ is conservative. Hence, $\D(R)$ is conservative by \Cref{Dcons}. Moreover, the $\infty$-category $\D(\B)$ admits all colimits by \cite[Proposition 1.3.5.21(1)]{HA}, and $\D(R)$ preserves all colimits by definition.
\end{proof}
    \bibliographystyle{alpha}
\bibliography{Quellen}
	
\end{document}